%% file: arxiv_dense.tex
\documentclass[reqno,12pt]{amsart}

\usepackage{psfig}

\usepackage{graphicx}
\usepackage{amssymb,amsmath,amstext}
\usepackage{latexsym}

\usepackage[norelsize]{algorithm2e}
\usepackage{color}
\usepackage{algorithmic}
\renewcommand{\algorithmcfname}{ALGORITHM}

\newtheorem{theorem}{Theorem}

\newtheorem{corollary}{Corollary}
\begin{document}

\title[A Dense Initialization for Limited-Memory Quasi-Newton Methods]
{A Dense Initialization for Limited-Memory Quasi-Newton Methods}

%\thanks{This research is support in part by National Science Foundation grants
%CMMI-1333326 and CMMI-1334042.}

\author[J. Brust]{Johannes Brust}
\email{jbrust@ucmerced.edu}
\address{Applied Mathematics, University of California, Merced, Merced, CA 95343}

\author[O. Burdakov]{Oleg Burdakov}
\email{oleg.burdakov@liu.se}
\address{Department of Mathematics, Link\"{o}ping University, SE-581 83 Link\"{o}ping, Sweden}

\author[J. Erway]{Jennifer B. Erway}
\email{erwayjb@wfu.edu}
\address{Department of Mathematics, Wake Forest University, Winston-Salem, NC 27109}

\author[R. Marcia]{Roummel F. Marcia}
\email{rmarcia@ucmerced.edu}
\address{Applied Mathematics, University of California, Merced, Merced, CA 95343}

\thanks{J.~B. Erway is supported in part by National Science Foundation grants
CMMI-1334042 and IIS-1741264}
\thanks{R.~F. Marcia is supported in part by National Science Foundation grants
CMMI-1333326 and IIS-1741490}

\date{\today}


\begin{abstract}
  We consider a family of dense initializations for
  limited-memory quasi-Newton methods.  
  %The proposed initialization
  %uses two parameters to approximate the curvature{\color{blue}s} of the 
  %{\color{blue}objective function} %Hessian 
  %in two complementary subspaces.  
  The proposed initialization %separates the full
%space 
exploits an eigendecompo\-sition-based separation of the full space
into two complementary subspaces, 
assigning a different initialization parameter to
each subspace. This family of dense
  initializations is proposed in the context of 
  a limited-memory Broyden-Fletcher-Goldfarb-Shanno ({L-BFGS}) 
  trust-region method that makes use of a shape-changing norm to
  define each subproblem. 
As with {L-BFGS} methods that traditionally
  use diagonal initialization, the dense initialization and the
  sequence of generated quasi-Newton matrices are never explicitly
  formed.  Numerical experiments on the CUTEst test set suggest that
  this initialization together with the shape-changing trust-region
  method outperforms other L-BFGS
  methods for solving general nonconvex unconstrained optimization
  problems.  While this dense initialization is proposed in the
  context of a special trust-region method, it has broad applications
  for more general quasi-Newton trust-region and line search methods.
  In fact, this initialization is suitable for use with any
  quasi-Newton update that admits a compact representation and, in
  particular, any member of the Broyden class of updates.
\end{abstract}

\maketitle

\newcommand{\mgap}{\;\;}
\newcommand{\bgap}{\;\;\;}
\newcommand{\qDef}{{\mathcal Q}}
\newcommand{\defined}{\mathop{\,{\scriptstyle\stackrel{\triangle}{=}}}\,}
\newcommand{\diag}{\text{diag}}
\renewcommand{\algorithmcfname}{ALGORITHM}

\makeatletter
\newcommand{\minimize}[1]{{\displaystyle\minim_{#1}}}
\newcommand{\minim}{\mathop{\operator@font{minimize}}}
\newcommand{\subject}{\mathop{\operator@font{subject\ to}}}  
\newcommand{\words}[1]{\mgap\text{#1}\mgap}
\def\BFGS{{\small BFGS}}
\def\LBFGS{{\small L-BFGS}}
\def\LSR{{\small L-SR1}}
\def\SR{{\small SR1}}
\def\CG{{\small CG}}
\def\DFP{{\small DFP}}
\def\OBS{{\small OBS}}
\def\OBSSC{{\small OBS-SC}}
\def\SCSR1{{\small SC-SR1}}
\def\PSB{{\small PSB}}
\def\QR{{\small QR}}
\def\MATLAB{{\small MATLAB}}
\renewcommand{\algorithmcfname}{ALGORITHM}
\renewcommand{\vec}[1]{#1}

\makeatother

\pagestyle{myheadings}
\thispagestyle{plain}
\markboth{J. BRUST, O. BURDAKOV, J. B. ERWAY AND R. F. MARCIA}{Dense initializations for limited-memory quasi-Newton methods}

\section{Introduction}
\label{sec:intro}
\input{1-Introduction}

\section{Background}
\label{sec:background}
\input{2-Background}

\section{The Proposed Method}
\label{sec:method}
\input{3-Method}

\section{Numerical Experiments}
\label{sec:numexp}
\input{4-NumericalExperiments}

\section{Conclusion}
\input{5-Conclusion}

%%%%%%%%%%%%%%%%%%%%%%%%%%%%%%%%%%%%%%%%%%%%%%
%\section{Acknowledgments}
%%%%%%%%%%%%%%%%%%%%%%%%%%%%%%%%%%%%%%%%%%%%%
%This research is support in part by National Science Foundation grants
%CMMI-1333326 and CMMI-1334042.

\bibliographystyle{abbrv}
\bibliography{myrefs}

\end{document}

%% file: 1-Introduction.tex
\label{sec-intro}
In this paper we propose a new
dense initialization for quasi-Newton methods 
to solve problems of the form
\begin{equation*}
	\underset{ \vec{x} \in \Re^n }{\text{ minimize }} f(\vec{x}),
\end{equation*}
where $ f:\Re^n\rightarrow\Re $ is %a general nonconvex
at least a continuously differentiable function, which is not necessarily convex.
The dense
initialization matrix is designed to be updated each time a new
quasi-Newton pair is computed (i.e., as often as once an iteration);
however, in order to retain the efficiency of limited-memory
quasi-Newton methods, the dense initialization matrix and the
generated sequence of quasi-Newton matrices are not explicitly formed.
This proposed initialization makes use of a partial eigendecomposition 
of these matrices for separating  %based separation of 
$\Re^n$ into two orthogonal subspaces -- one for which
there is approximate curvature information and the other for which
there is no reliable curvature information. 
%\rfm{Should we just delete this sentence: } A different scaling
%parameter may then be used for each subspace.  
This initialization has
broad applications for general quasi-Newton trust-region and line
search methods.  In fact, this work can be applied to any quasi-Newton
method that uses an update with a compact representation,
which includes any member of the Broyden class of updates. For this
paper, we explore its use in one specific algorithm; in particular
we consider a limited-memory Broyden-Fletcher-Goldfarb-Shanno ({\small
L-BFGS}) trust-region method where each subproblem is defined using a
shape-changing norm~\cite{BurdakovLMTR16}.  The reason for this choice
is that the dense initialization is naturally well-suited for solving
{\small L-BFGS} trust-region subproblems defined by this norm.
Numerical results on the {\small CUTE}st test set suggest that the dense initialization outperforms
other \LBFGS{} methods.

\medskip 

The {\small BFGS} update is the most widely-used quasi-Newton update for large-scale optimization;
it is defined by the recursion formula 
\begin{equation}
	\label{eq:recursion}
  \vec{B}_{k+1} = \vec{B}_{k} -
\frac{1 }{\vec{s}^T_{k}\vec{B}_{k}\vec{s}_{k}}\vec{B}_{k}\vec{s}_{k}
\vec{s}_{k}^T\vec{B}_{k}+
\frac{1}{\vec{s}^T_{k}\vec{y}_{k}} \vec{y}_{k}\vec{y}^T_{k},
\end{equation}
where
\begin{equation}\label{eqn-sy}
	\vec{s}_{k} \defined \vec{x}_{k+1} - \vec{x}_{k} \quad \text{ and } \quad \vec{y}_{k} \defined \nabla f(\vec{x}_{k+1}) - \nabla f(\vec{x}_{k}),
\end{equation}
and $ \vec{B}_0 \in \Re^{n \times n} $ is a suitably-chosen
initial matrix.
%\je{which in general depends on $k$.  For the purpose
%Tof notational simplicity, we will drop the dependence of the initial matrix on T$k$.}
%\je{(Note that $B_0$ may change as $k$ is updated, and thus,
%is sometimes written as $B_0^{(k)}$; however, for the purposes of
%notational simplicitly we will drop the superscript indicating the
%dependence on $k$ while still assuming that $B_0$ may change each iteration.)}
This rank-two update to $B_{k}$ preserves positive definiteness when $
\vec{s}^T_{k}\vec{y}_{k} > 0$.

{\small L-BFGS} is a limited-memory variant of {\small BFGS} that only
stores a predetermined number, $m$, of the most recently-computed pairs $\{s_i,y_i\}$
where $m\ll n$. (Typically, $m\in[3,7]$ (see, e.g.,~\cite{ByrNS94}).)  Together with an intial matrix $B_0$ that depends on $k$, 
these pairs are used to compute $B_k$.  For notational simplicity, we drop the
dependence of the initial matrix on $k$ and simply denote it as $B_0$.
 This limitation on the number of stored pairs
allows for a practical implementation of the {\small BFGS} method
% makes {\small L-BFGS} a practical method
for large-scale optimization.

There are several desirable properties for picking the initial matrix
$B_0$.  First, in order for the sequence $\{B_k\}$ generated by
(\ref{eq:recursion}) to be symmetric and positive definite, it is necessary
that $ \vec{B}_0 $ is symmetric and positive definite. Second, it is
desirable for $B_0$ to be easily invertible so that solving linear systems with any matrix
in the sequence is computable using the so-called ``two-loop
recursion''~\cite{ByrNS94} or other recursive formulas for $B_k^{-1}$ (for
an overview of other available methods see~\cite{ErwayMarcia17LAA}).  For
these reasons, $\vec{B}_0$ is often chosen to be a scalar multiple of the
identity matrix, i.e., 
\begin{equation}\label{eqn-diagInit}
	\vec{B}_0 =  \gamma_k \vec{I}, \quad
\text{with}\quad \gamma_k > 0. 
\end{equation}
For {\small BFGS} matrices, the conventional choice for the initialization parameter $ \gamma_k $ is 
\begin{equation}\label{eqn-B0-usual} \gamma_k = \frac{\vec{y}^T_{k}
  \vec{y}_{k} }{\vec{s}^T_{k} \vec{y}_{k}},
\end{equation} which can be viewed as a spectral
estimate for $\nabla^2 f( \vec{x}_k )$~\cite{NocW99}.
(This choice was originally proposed in~\cite{ShannoPh78} using a derivation
based on
optimal conditioning.)
It is worth noting that this choice of $ \gamma_k $ can also be derived as the minimizer
of the scalar minimization problem 
\begin{equation}\label{eqn-mingamma}
 \gamma_k = \underset{ \gamma }{ \text{
   argmin } } \left\| \vec{B}^{-1}_0 \vec{y}_k - \vec{s}_k \right\|^2_2,
\end{equation}
where $ \vec{B}^{-1}_0 = \gamma^{-1}\vec{I} $.
For numerical studies on this choice of initialization,
see, e.g., the references listed within~\cite{BurWX96}.

\bigskip

In this paper, we consider a specific dense initialization in lieu of
the usual diagonal initialization.  The aforementioned separation of
$\Re^n$ into two orthogonal subspaces allows for different initialization parameters
to be used to estimate the curvature of the underlying %Hessian 
function in
these subspaces. In one space 
(the space spanned by the most recent updates $\{s_i, y_i\}$ with $k-m \le i \le k-1$), 
estimates of the curvature of the
underlying function are available, and thus, one initialization parameter can be set
using this information.  However, in its orthogonal complement,
curvature information is not available.  
Therefore, if the component of the trial step in the orthogonal subspace is (relatively) too large,
the predictive quality of the whole trial step  is expected to deteriorate. As a result, the
trust-region radius might be reduced, despite the fact that the predictive quality of the component in the aforementioned small subspace may be sufficiently good.
Separating the whole space  into these two subspaces allows users to treat each subspace
differently.
%no curvature information is available, and thus, a separate choice can be made
%for the curvature.  This allows users to treat the two subspaces
%differently.
An alternative view of this initialization is that it
makes use of \emph{two} spectral estimates of $\nabla^2 f(x_k)$.
Finally, the proposed initialization also allows for efficiently
solving and computing products with the resulting quasi-Newton
matrices.

\medskip

The paper is organized in five sections.  In Section 2, we review
properties of {\small L-BFGS} matrices arising from their special recursive
structure as well as overview the shape-changing trust-region method to be
used in this paper.  In Section 3, we present the proposed trust-region
method that uses a shape-changing norm together with a dense initialization
matrix.  While this dense initialization is presented in one specific
context, it can be used in combination with any quasi-Newton update that
admits a so-called \emph{compact representation}.  Numerical experiments
comparing this method with other combinations of initializations and
\LBFGS{} methods are reported in Section 4, and concluding remarks  are
found in Section 5.

%% file: 2-Background.tex
In this section, we overview the compact formulation for \LBFGS{} matrices
and how to efficiently compute a partial eigendecomposition.  Finally, we
review the shape-changing trust-region method considered in this paper.

\subsection{The compact representation}  
The special structure of the recursion formula for \LBFGS{} matrices
admits a so-called compact representation~\cite{ByrNS94},
which is overviewed in this section.

Using the $m$ most recently computed pairs $\{s_j\}$ and $\{y_j\}$ given
in (\ref{eqn-sy}), we define the following matrices
\begin{equation*}
  \vec{S}_k \defined \left[ \vec{s}_{k-m} \,\, \cdots \,\, \vec{s}_{k-1}\right] \quad \text{ and } \quad \vec{Y}_k \defined \left[ \vec{y}_{k-m} \,\, \cdots \,\, \vec{y}_{k-1}\right].
\end{equation*}
With $\vec{L}_k$ taken to be the strictly lower triangular part of the matrix
of $ \vec{S}^T_k \vec{Y}_k
$,
and $ \vec{D}_k $ defined as the diagonal of $ \vec{S}^T_k \vec{Y}_k
$,
the compact representation of an \LBFGS{} matrix is
\begin{equation}
	\label{eq:comactlbfgs}
	\vec{B}_k = \vec{B}_0 + \vec{\Psi}_k \vec{M}_k \vec{\Psi}^T_k,
\end{equation}
where 
\begin{equation}\label{eqn-PsiM}
\Psi_k\defined
 \left[ \vec{B}_0\vec{S}_k \,\, \vec{Y}_k\right] \quad
\text{and} \quad
M_k\defined 
	-\left[ 
\begin{array}{c c}
\vec{S}^T_k\vec{B}_0 	\vec{S}_k 			& \vec{L}_k \\
	\vec{L}^T_k 								& -\vec{D}_k
\end{array}
\right]^{-1}
\end{equation}
(see~\cite{ByrNS94} for details).  Note that $ \vec{\Psi}_k \in \Re^{n
  \times 2m} $, and $ \vec{M}_k \in \Re^{2m \times 2m} $ is invertible
provided $s_i^Ty_i > 0$ for all $i$~\cite[Theorem 2.3]{ByrNS94}.  An
advantage of the compact representation is that if $ \vec{B}_0 $ is
chosen to be a multiple of the identity, then computing products
with $B_k$ or solving linear systems with $B_k$ can be done efficiently~\cite{ErwayMarcia17LAA,LukV13}.

It should be noted that {\small L-BFGS} matrices are just one member of
the Broyden class of matrices (see, e.g.,~\cite{NocW99}),
and in fact every member of the Broyden class of matrices admits a
compact representation~{\cite{DeGuchyEM16,ErwayM15,LukV13}.

\subsection{Partial eigendecomposition of $ \vec{B}_k $}
\label{subsec:eigen}
If $ \vec{B}_0$ is taken to be a multiple of the identity matrix, then the partial
eigendecomposition of $B_k$ can be computed efficiently from the compact
representation \eqref{eq:comactlbfgs} using either a partial {\small QR}
decomposition~\cite{BurdakovLMTR16} or a partial singular
value decomposition ({\small SVD})~\cite{Lu92}.  Below, we review the
approach that uses the {\small QR} decomposition, and we assume
that $ \vec{\Psi}_k $ has rank $ r = 2m $.  (For the rank-deficient case,
see the techniques found in \cite{BurdakovLMTR16}.)

Let
\begin{equation*}
	\vec{\Psi}_k = \vec{Q}\vec{R}, 
\end{equation*}
be the so-called ``thin'' {\small QR} factorization of $\Psi_k$, where
$\vec{Q} \in \Re^{n\times r} $ and $ \vec{R} \in \Re^{r\times
  r} $.  Since the matrix $ \vec{R} \vec{M}_k \vec{R}^T$  is a small $(r\times r)$
  matrix with $r\ll n$ (recall that $r=2m$, where $m$ is typically between 3 and 7), it is
computationally feasible to calculate its eigendecomposition; thus, suppose
$\vec{W} \hat{\vec{\Lambda}} \vec{W}^T $ is the eigendecomposition of $
\vec{R} \vec{M}_k \vec{R}^T.$ Then,
\begin{equation*}
	\vec{\Psi}_k\vec{M}_k\vec{\Psi}^T_k = \vec{Q} \vec{R} \vec{M}_k\vec{R}^T \vec{Q}^T= \vec{Q} \vec{W} \hat{\vec{\Lambda}}\vec{W}^T \vec{Q}^T = 
									\vec{\Psi}_k \vec{R}^{-1} \vec{W} \hat{\vec{\Lambda}}\vec{W}^T  \vec{R}^{-T} \vec{\Psi}^{T}_k.
\end{equation*}
Defining
\begin{equation}\label{eqn-pparallel}
\vec{P}_{\parallel} = \vec{\Psi}_k \vec{R}^{-1} \vec{W},
\end{equation}
gives that
\begin{equation}\label{eqn-lambdahat}
	\vec{\Psi}_k \vec{M}_k \vec{\Psi}^T_k = \vec{P}_\parallel \hat{\vec{\Lambda}} \vec{P}^T_\parallel.
\end{equation}
Thus, for $B_0=\gamma_kI$, the eigendecomposition of $B_k$ can be written as
\begin{equation}
	\label{eq:eiglbfgs}
	\vec{B}_k = \gamma_k \vec{I} + \vec{\Psi}_k \vec{M}_k \vec{\Psi}^T_k = P\Lambda P^T,
\end{equation}
where 
\begin{equation} \label{eqn-P}
  P \defined
  \left[ \vec{P}_{\parallel} \,\, \vec{P}_{\perp}\right], \quad
  \Lambda \defined
\left[ \begin{array}{cc}
    \hat{\vec{\Lambda}}+\gamma_k I_r 	& 			\\
    & \gamma_k I_{n-r}
  \end{array}
\right],
\end{equation}
and $ \vec{P}_{\perp} \in
\mathbb{R}^{n \times (n-r)} $ is defined as the orthogonal complement of $
\vec{P}_{\parallel} $, i.e., $ \vec{P}^T_{\perp}\vec{P}_{\perp} =
\vec{I}_{n-r} $ and $ \vec{P}^T_{\parallel}\vec{P}_{\perp} = \vec{0}_{r
  \times (n-r)}$ .  
Hence, $\vec{B}_k$ has $r$ eigenvalues given by the diagonal
elements of $ \hat{ \vec{\Lambda} } +\gamma_k \vec{I}_{r} $ and
the remaining eigenvalues are $ \gamma_k $ with multiplicity $ n-r $.

\subsubsection{Practical computations}
Using the above method yields the eigenvalues of $B_k$ as well as the
ability to compute products with $P_\parallel$. 
Formula (\ref{eqn-pparallel}) indicates that $Q$ is not
required to be explicitly formed in order to compute products with
$P_\parallel$.
For this reason, it is desirable to avoid
forming $Q$ by computing only $R$ via the Cholesky factorization of
$ \vec{\Psi}^T_k\vec{\Psi}_k$,
i.e., $ \vec{\Psi}^T_k\vec{\Psi}_k = \vec{R}^T\vec{R} $
(see~\cite{BurdakovLMTR16}).

At an additional expense,
the eigenvectors stored in the columns of $P_\parallel$ may be formed and
stored.  For the shape-changing trust-region method used in this paper, it
is not required to store $P_\parallel$.
In contrast, the matrix $P_\perp$ is
prohibitively expensive to form.  It turns out that for this work it is
only necessary to be able to compute projections into the subspace $
\vec{P}_{\perp} \vec{P}_{\perp}^T $, which can be done using the
identity
\begin{equation} \label{eqn-projection}
\vec{P}_{\perp}\vec{P}^T_{\perp} = \vec{I}-
\vec{P}_{\parallel}\vec{P}^T_{\parallel}.
\end{equation}

\subsection{A shape-changing L-BFGS trust-region method}
Generally speaking,
at
 the $k$th step of a trust-region method, a search direction is computed by
approximately solving the trust-region subproblem
\begin{equation}
	\label{eq:subproblem}
		\vec{p}^* = \underset{\left\| \vec{p} \right\| \le {\Delta_k}}{\text{ argmin }} Q(\vec{p}) \defined\vec{g}^T_k \vec{p} + \frac{1}{2} \vec{p}^T \vec{B}_k \vec{p}, 
\end{equation}
where $ \vec{g}_k \defined \nabla f(\vec{x}_k) $, $ \vec{B}_k \approx
\nabla^2 f(\vec{x}_k) $, and $ \Delta_k > 0 $ is the trust-region radius.
When second derivatives are unavailable or computationally too expensive to
compute, approximations using gradient information may be preferred.  Not
only do quasi-Newton matrices use only gradient and function information,
but in the large-scale case, these Hessian approximations are never stored;
instead, a recursive formula or methods that avoid explicitly forming $B_k$
may be used to compute matrix-vector products with the approximate Hessians
or their inverses~\cite{ByrNS94,ErwayM15,ErwayMarcia17LAA,LukV13}.
 
\medskip

Consider the trust-region subproblem defined by the shape-changing infinity norm:
\begin{equation}
	\label{eq:subprobsc}
	\underset{\left\| \vec{p} \right\|_{\vec{P},\infty} \le {\Delta_k}}{\text{ minimize }} Q(\vec{p}) = \vec{g}^T_k \vec{p} + \frac{1}{2} \vec{p}^T \vec{B}_k \vec{p},
\end{equation}
where 
\begin{equation}\label{eq:shape-changing_norm}
	\left\| \vec{p} \right\|_{\vec{P},\infty} \defined \text{max}\left( \| \vec{P}^T_{\parallel} \vec{p} \|_{\infty}, \| \vec{P}^T_{\perp} \vec{p} \|_{2} \right)
\end{equation}
and $ \vec{P}_{\parallel} $ and $ \vec{P}_{\perp} $ are given in
(\ref{eqn-P}).  Note that the ratio $\|p\|_2/\|p\|_{\vec{P},\infty}$
  does not depend on $n$ and only moderately depends on $r$. (In particular, $1 \le \|p\|_2/\|p\|_{\vec{P},\infty} \le \sqrt{r+1}$.)
  Because this norm depends on the
eigenvectors of $ \vec{B}_k $, the shape of the trust region changes each
time the quasi-Newton matrix is updated, which is possibly every iteration
of a trust-region method.  (See~\cite{BurdakovLMTR16} for more details and
other properties of this norm.)  The motivation for this choice of norm is
that the the trust-region subproblem (\ref{eq:subprobsc}) decouples into
two separate problems for which closed-form solutions exist.

We now review the closed-form solution to (\ref{eq:subprobsc}), as detailed 
in~\cite{BurdakovLMTR16}.  Let 
\begin{equation}\label{eqn-varchange}
  \vec{v} = \vec{P}^T\vec{p} = 	\left[ 
										\begin{array}{c}
                                                                                  \vec{P}^T_{\parallel}\vec{p} \\
                                                                                  \vec{P}^T_{\perp}\vec{p} 
										\end{array}
                                                                              \right]
                                                                              \defined 	\left[ 
										\begin{array}{c}
											\vec{v}_{\parallel} \\
											\vec{v}_{\perp} 
										\end{array}
									\right] \quad \text{ and } \quad
				\vec{P}^T\vec{g}_k = 	\left[ 
										\begin{array}{c}
											\vec{P}^T_{\parallel}\vec{g}_k \\
											\vec{P}^T_{\perp}\vec{g}_k 
										\end{array}
									\right]
								\defined 	\left[ 
										\begin{array}{c}
											\vec{g}_{\parallel} \\
											\vec{g}_{\perp} 
										\end{array}
									\right].
\end{equation}
With this change of variables, the objective function of (\ref{eq:subprobsc})
becomes
\begin{align*}
	Q\left( \vec{P}\vec{v} \right) &= \vec{g}^T_k \vec{P}\vec{v} + \frac{1}{2} \vec{v}^T \left(  \hat{\vec{\Lambda}}+\gamma_k \vec{I}_n   \right) \vec{v}  \\
							&= \vec{g}^T_{\parallel} \vec{v}_{\parallel} + \vec{g}^T_{\perp} \vec{v}_{\perp} +
									\frac{1}{2}\left( \vec{v}^T_{\parallel} \left( \hat{\vec{\Lambda}} + \gamma_k \vec{I}_r  \right) \vec{v}_{\parallel} + 
										\gamma_k\left\| \vec{v}_{\perp} \right\|^2_2  \right) \\
							&= \vec{g}^T_{\parallel} \vec{v}_{\parallel}+\frac{1}{2} \vec{v}^T_{\parallel} \left( 
\hat{\vec{\Lambda}} + \gamma_k \vec{I}_r\right) \vec{v}_{\parallel} + \vec{g}^T_{\perp} \vec{v}_{\perp} + \frac{1}{2}\gamma_k\left\| \vec{v}_{\perp} \right\|^2_2.
\end{align*}
The trust-region constraint 
$\left\| \vec{p} \right\|_{\vec{P},\infty} \le \Delta_k $ implies $ \left\|
  \vec{v}_{\parallel} \right\|_{\infty} \le \Delta_k $ and $ \left\|
  \vec{v}_{\perp} \right\|_{2} \le \Delta_k$, which decouples
\eqref{eq:subprobsc} into the following two trust-region subproblems:
\begin{eqnarray}
  \underset{ \left\| \vec{v}_{\parallel} \right\|_{\infty} \le \Delta_k }{\text{ minimize }} q_{\parallel}\left(\vec{v}_{\parallel}\right) & \defined & \vec{g}^T_{\parallel} \vec{v}_{\parallel} + \frac{1}{2} \vec{v}^T_{\parallel} \left(  \hat{\vec{\Lambda}} + \gamma_k \vec{I}_r \right) \vec{v}_{\parallel} \label{eqn-sub1} \\
	\underset{ \left\| \vec{v}_{\perp} \right\|_{2} \le \Delta_k }{\text{ minimize }} q_{\perp}\left(\vec{v}_{\perp}\right) &\defined& \vec{g}^T_{\perp} \vec{v}_{\perp} + 
	\frac{1}{2}\gamma_k\left\| \vec{v}_{\perp} \right\|^2_2. \label{eqn-sub2}
\end{eqnarray}
Observe that the resulting minimization problems are considerably simpler
than the original problem since in both cases the Hessian of the new
objective functions are diagonal matrices.  The solutions to
these decoupled problems have closed-form analytical solutions~\cite{BurdakovLMTR16,BruEM16}. Specifically, letting 
$ \lambda_i \defined \hat{\lambda}_{i} +
\gamma_k$, the solution to (\ref{eqn-sub1}) is given coordinate-wise by 
\begin{equation}\label{eqn:solution_vparallel}
		[\vec{v}^*_{||} ]_i =
		\begin{cases}
		-\frac{\left[ \vec{g}_{||}\right]_i}{ \lambda_i} 
			& \text{ if } \left| \frac{ \left[ \vec{g}_{||}\right]_i  }{\lambda_i} \right| \le \Delta_k \text{ and }  \lambda_i > 0, \\
		 c & \text{ if } \left[ \vec{g}_{\parallel}\right]_i = 0 \text{ and }  \lambda_i = 0,\\
		- \text{sgn}(\left [ \vec{g}_{\parallel} \right ]_i) \Delta_k											
		& \text{ if } \left[ \vec{g}_{\parallel}\right]_i \ne 0 \text{ and }  \lambda_i = 0,\\				
		\pm\Delta_k 											
		& \text{ if } \left[ \vec{g}_{\parallel}\right]_i = 0 \text{ and }  \lambda_i < 0,\\
		-\frac{\Delta_k}{ \left| \left[ \vec{g}_{||}\right]_i \right|} \left[ \vec{g}_{||}\right]_i 
		& \text{ otherwise},
		\end{cases}, 
\end{equation}
where $c$ is any real number in $[-\Delta_k, \Delta_k]$ and `sgn' denotes the
signum function. Meanwhile, the minimizer of (\ref{eqn-sub2})
is given by
\begin{equation}
\label{eq:subsolnperp}
\vec{v}^*_{\perp} = \beta \vec{g}_{\perp},
\end{equation}
where
\begin{equation}
\label{eq:subsolnbeta}
\beta =
\begin{cases}			
	-\frac{1}{\gamma_k} 							& \text{ if } \gamma_k > 0 \text{ and }  \left \| \vec{g}_{\perp} \right \|_2 \le \Delta_k |\gamma_k|, \\
-\frac{ \Delta_k}{\| \vec{g}_{\perp} \|_2}     & \text{ otherwise. }			
\end{cases}
\end{equation}
Note that the solution to (\ref{eq:subprobsc}) is then
\begin{equation}\label{eqn-finalsolution}
  \vec{p}^* = \vec{P}\vec{v}^* = \vec{P}_{\parallel}\vec{v}^*_{\parallel} + \vec{P}_{\perp}\vec{v}^*_{\perp} = 
\vec{P}_{\parallel}\vec{v}^*_{\parallel} + \beta\vec{P}_{\perp}g_{\perp} =
\vec{P}_{\parallel}\vec{v}^*_{\parallel} + \beta\vec{P}_{\perp}\vec{P}_{\perp}^Tg_{k},
\end{equation}
where the latter term is computed using (\ref{eqn-projection}).
Additional simplifications yield the following expression for $p^*$:
\begin{equation}\label{eq:pstar}
	p^* = \beta g + P_{\parallel}(v_{\parallel}^* - \beta g_{\parallel}).
\end{equation}	
The overall cost of computing the solution to  \eqref{eq:subprobsc} is comparable to that of using the Euclidean norm
(see~\cite{BurdakovLMTR16}).  The main advantage of using the shape-changing norm \eqref{eq:shape-changing_norm}
is that the solution $p^*$ in \eqref{eq:pstar} has a closed-form expression.

%% file: 3-Method.tex
In this section, we present a new dense initialization and demonstrate
how it is naturally well-suited for trust-region methods defined by
the shape-changing infinity norm. Finally, we present a full
trust-region algorithm that uses the dense initialization, consider
its computational cost, and prove global convergence.

\subsection{Dense initial matrix $ \widehat{\vec{B}}_0 $}
\label{subsec:denseinitial}
In this section, we propose a new dense initialization for
quasi-Newton methods.  Importantly, in order to retain the efficiency
of quasi-Newton methods the dense initialization matrix and
subsequently updated quasi-Newton matrices are never explicitly
formed.  This initialization can be used with any quasi-Newton update
for which there is a compact representation; however, for simplicity,
we continue to refer to the {\small BFGS} update throughout this
section.  For notational purposes, we use the initial matrix $B_0$ to
represent the usual initialization
and $\widehat{\vec{B}}_0$ to denote the proposed dense initialization.
Similarly, $\{B_k\}$ and $\{\widehat{\vec{B}}_k\}$ will be used to
denote the sequences of matrices obtained using the
initializations $B_0$ and $\widehat{\vec{B}}_0$, respectively.

Our goal in choosing an alternative initialization is
four-fold:
(i) to be able to treat subspaces differently depending on whether
curvature information is available or not,
(ii) to preserve properties of symmetry and positive-definiteness,
(iii) to be able to efficiently compute products with the resulting quasi-Newton matrices,
and
(iv) to be able to efficiently solve linear systems 
involving the resulting quasi-Newton matrices.  The initialization proposed
in this paper leans upon two different parameter choices that can be viewed
as an 
estimate of the curvature of $\nabla^2 f(x_k)$ in two subspaces:
one spanned by the columns
of $P_\parallel$ and another 
 spanned by the columns of $P_\perp$.

The usual initialization for a {\small BFGS} matrix $B_k$ is $B_0=\gamma_k I$,
where $\gamma_k>0$.
Note that this initialization is equivalent to
$$	\vec{B}_0 = \gamma_kPP^T=\gamma_k\vec{P}_{\parallel} \vec{P}^T_{\parallel} + \gamma_{k} \vec{P}_{\perp} \vec{P}^T_{\perp}.$$
In contrast, for a given $\gamma_k, \gamma^\perp_k \in \Re$, consider the
following symmetric, and in general, dense initialization matrix:
\begin{equation}\label{eq:denseB0}
	\widehat{\vec{B}}_0 = \gamma_k\vec{P}_{\parallel} \vec{P}^T_{\parallel} + \gamma^\perp_k \vec{P}_{\perp} \vec{P}^T_{\perp},
\end{equation} 
where $P_\parallel$ and $P_\perp$ are the matrices of eigenvectors 
defined in Section~\ref{subsec:eigen}.  We now derive the
eigendecomposition of $\widehat{B}_k$.

\begin{theorem}\label{eqn-thm1}
Let $\widehat{\vec{B}}_0$ be defined as in \eqref{eq:denseB0}.  
Then $\widehat{\vec{B}}_k$ generated using (\ref{eq:recursion})
has the eigendecomposition
\begin{equation}\label{eq:Bkhat_eig}
	\widehat{\vec{B}}_k 
	 =
	\left[ \vec{P}_{\parallel} \,\, \vec{P}_{\perp}\right]
	\left[ 
		\begin{array}{c c}
			\hat{\vec{\Lambda}} +\gamma_k \vec{I}_r 	& 			\\
														&	\gamma^\perp_k\vec{I}_{n-r}
		\end{array}
	\right]
	\left[ \vec{P}_{\parallel} \,\,  \vec{P}_{\perp}\right]^T,
\end{equation}
where $P_\parallel, P_\perp,$ and $\hat{\Lambda}$ are given in
(\ref{eqn-pparallel}), (\ref{eqn-P}), and (\ref{eqn-lambdahat}), respectively.
\end{theorem}
\begin{proof}
First note that the columns of $S_k$ are in
$\text{Range}(\Psi_k)$, where $\Psi_k$ is defined in (\ref{eqn-PsiM}).  From
(\ref{eqn-pparallel}), $\text{Range} (\vec{\Psi}_k) = \text{Range}
(\vec{P}_{\parallel})$; 
thus,
$P_\parallel P_\parallel^T S_k = S_k$ and $P_\perp^TS_k = 0$.
This gives that
\begin{equation}\label{eqn-B0Bhat}
	\widehat{\vec{B}}_0 \vec{S}_k 
	= \gamma_k \vec{P}_{\parallel}\vec{P}^T_{\parallel} \vec{S}_k 
	+ \gamma^\perp_k \vec{P}_{\perp} \vec{P}^T_{\perp}\vec{S}_k
	= \gamma_k \vec{S}_k 
	= \vec{B}_0 \vec{S}_k.
\end{equation}
Combining the compact representation of $\widehat{B}_k$ ((\ref{eq:comactlbfgs})
and (\ref{eqn-PsiM}))
together with (\ref{eqn-B0Bhat}) yields
\begin{eqnarray*}
	\widehat{\vec{B}}_k 
	&=&  
	\widehat{\vec{B}}_0 - 
	\left[ 
		\widehat{\vec{B}}_0\vec{S}_k \,\, \vec{Y}_k 
	\right] 
	\left[ 
	\begin{array}{c c}
		\vec{S}^T_k	\widehat{\vec{B}}_0 \vec{S}_k 		& \vec{L}_k \\
		\vec{L}^T_k 								& -\vec{D}_k
	\end{array}
	\right]^{-1}
	\left[ 
	\begin{array}{c}
		\vec{S}^T_k	\widehat{\vec{B}}_0	\\
		\vec{Y}^T_k
	\end{array}
	\right ] \\
	&=&  
	\widehat{\vec{B}}_0 - 
	\left[ 
		{\vec{B}}_0\vec{S}_k \,\, \vec{Y}_k 
	\right] 
	\left[ 
	\begin{array}{c c}
		\vec{S}^T_k{\vec{B}}_0 \vec{S}_k 		& \vec{L}_k \\
		\vec{L}^T_k 								& -\vec{D}_k
	\end{array}
	\right]^{-1}
	\left[ 
	\begin{array}{c}
		\vec{S}^T_k{\vec{B}}_0	\\
		\vec{Y}^T_k
	\end{array}
	\right ] \\
	&=&
	 \gamma_k\vec{P}_{\parallel} \vec{P}^T_{\parallel} + \gamma^\perp_k \vec{P}_{\perp} \vec{P}^T_{\perp} 
	 +
	 \vec{P}_{\parallel}  \hat{\vec{\Lambda}} \vec{P}^T_{\parallel}
	 \\
	 &=&
	  \vec{P}_{\parallel} \left( \hat{\Lambda} + \gamma_k \vec{I}_r \right) \vec{P}^T_{\parallel} + \gamma^\perp_k \vec{P}_{\perp}\vec{P}^T_{\perp},
\end{eqnarray*}
which is equivalent to
\eqref{eq:Bkhat_eig}. $\square$
\end{proof}

\medskip

It can be easily verified that (\ref{eq:Bkhat_eig}) holds also for
$P_\parallel$ defined in~\cite{BurdakovLMTR16} for possibly
rank-deficient $\Psi_k$.  (Note that (8) 
applies only to the special case when $\Psi_k$ is full-rank.)

\medskip

Theorem~\ref{eqn-thm1} shows that the matrix $\widehat{\vec{B}}_k$ that results from using
the initialization (\ref{eq:denseB0}) shares the same eigenvectors as
$B_k$, generated using $B_0=\gamma_k I$.  Moreover, the eigenvalues
corresponding to the eigenvectors stored in the columns of $P_\parallel$
are the same for $\widehat{B}_k$ and $B_k$.  The only difference
in the eigendecompositions of $\widehat{B}_k$ and $B_k$ is in the
eigenvalues corresponding to the eigenvectors stored in the columns of
$P_\perp$.  This is summarized in the following corollary.

\begin{corollary}\label{eqn-cor1}
Suppose $B_k$ is a {\small BFGS} matrix initialized with
  $B_0=\gamma_kI$ and $\widehat{\vec{B}}_k$ is a {\small BFGS} matrix
  initialized with (\ref{eq:denseB0}).  Then $B_k$ and
  $\widehat{\vec{B}}_k$ have the same eigenvectors; moreover, these
  matrices have $r$ eigenvalues in common given by
  $\lambda_i\defined\hat{\lambda}_i+\gamma_k$ where $\hat{\Lambda}=\diag(
  \hat{\lambda}_1, \ldots, \hat{\lambda}_r)$.  \end{corollary} 
\begin{proof} The corollary
follows immediately by comparing (\ref{eq:eiglbfgs}) with \eqref{eq:Bkhat_eig}. $\square$
\end{proof}

\medskip

The results of Theorem~\ref{eqn-thm1} and Corollary~\ref{eqn-cor1} may seem surprising at
first since every term in the compact representation
((\ref{eq:comactlbfgs}) and (\ref{eqn-PsiM})) depends on the
initialization; moreover, $\widehat{B}_0$ is, generally speaking, a dense
matrix while $B_0$ is a diagonal matrix.  However, viewed from the
perspective of (\ref{eq:denseB0}), the parameter $\gamma^\perp_k$ only plays
a role in scaling the subspace spanned by the columns of $P_\perp$.

\medskip

The initialization $\widehat{B}_0$ allows for two separate curvature
approximations for the {\small BFGS} matrix: one in the space spanned by
columns of $P_\parallel$ and another in the space spanned by the columns of
$P_\perp$. 
In the next subsection, we show that this initialization 
is naturally well-suited for solving trust-region
subproblems defined by the shape-changing infinity norm.

\subsection{The trust-region subproblem} \label{subsec-decoupled}
Here we will show that the use of 
$\widehat{B}_0$ provides the same subproblem separability as
$B_0$ does in the case of the shape-changing infinity norm.

\medskip

For $\widehat{\vec{B}}_0$ given by
\eqref{eq:denseB0}, consider the objective function of the trust-region
subproblem (\ref{eq:subprobsc}) resulting from the change of variables
(\ref{eqn-varchange}):
\begin{eqnarray*}
	Q ( \vec{P}\vec{v} )
	&= & 
g_k^TPv + \frac{1}{2}v^TP^T\widehat{B}_kPv \\
& = & 
\vec{g}^T_{\parallel} \vec{v}_{\parallel}
	+ \frac{1}{2} \vec{v}^T_{\parallel} \left( 
	\hat{\vec{\Lambda}} +\gamma_k \vec{I}_r   \right) \vec{v}_{\parallel} 
	+ \vec{g}^T_{\perp} \vec{v}_{\perp} 
	+ \frac{1}{2}\gamma^\perp_k \left\| \vec{v}_{\perp} \right\|^2_2.
\end{eqnarray*}
Thus, (\ref{eq:subprobsc}) decouples into two subproblems: The
corresponding subproblem for $q_{\parallel}(\vec{v}_{\parallel})$ remains (\ref{eqn-sub1}) and the subproblem for $q_{\perp}(\vec{v}_{\perp})$ becomes
\begin{equation}\label{eq:q_perp2}
	\underset{ \left\| \vec{v}_{\perp} \right\|_{2} \le \Delta_k }{\text{ minimize }} q_{\perp}\left(\vec{v}_{\perp}\right) \defined \vec{g}^T_{\perp} \vec{v}_{\perp} + 
	\frac{1}{2}\gamma^\perp_k \left\| \vec{v}_{\perp} \right\|^2_2.
\end{equation}
The solution to (\ref{eq:q_perp2}) is now given by
\begin{equation}
\label{eq:subsolnperp2}
\vec{v}^*_{\perp} = \widehat{\beta} \vec{g}_{\perp},
\end{equation}
where
\begin{equation}
\label{eq:subsolnbeta2}
\widehat{\beta} =
\begin{cases}			
	-\frac{1}{\gamma^\perp_k} 							& \text{ if } \gamma^\perp_k > 0 \text{ and }  \left \| \vec{g}_{\perp} \right \|_2 \le \Delta_k |\gamma^\perp_k|, \\
-\frac{ \Delta_k}{\| \vec{g}_{\perp} \|_2}     & \text{ otherwise. }			
\end{cases}
\end{equation}

Thus, the solution  
$p^*$ can be expressed as
\begin{equation}\label{eqn-Roummel}
	p^* = \widehat{\beta} g + P_{\parallel}(v_{\parallel}^* - \widehat{\beta} g_{\parallel}),
\end{equation}
which can computed as efficiently as the solution in \eqref{eq:pstar} for conventional initial matrices
since they differ only by the scalar 
($\widehat{\beta}$ in (\ref{eqn-Roummel}) versus $\beta$ in (\ref{eq:pstar})).

\subsection{Determining the parameter $\gamma^\perp_k$}\label{subsec:gammaperp}
The values $\gamma_k$ and $\gamma^\perp_k$ can be updated at each iteration.
Since we have little information about the underlying function $f$ in the
subspace spanned by the columns of $P_\perp$, it is reasonable to make
conservative (i.e., large) choices  for $\gamma^\perp_k$.  Note that in the case that
$\gamma^\perp_k > 0 \text{ and } \left \| \vec{g}_{\perp} \right \|_2 \le
\Delta_k |\gamma^\perp_k|$, the parameter $\gamma^\perp_k$ scales the
solution $v_\perp^*$ (see \ref{eq:subsolnbeta2}); thus, large values of
$\gamma^\perp_k$ will reduce these step lengths in the space spanned by
$P_\perp$.  Since the space $P_\perp$ does not explicitly use information
produced by past iterations, it seems desirable to choose $\gamma^\perp_k$
to be large.  However, the larger that $\gamma^\perp_k$ is chosen, 
the closer $v^*_\perp$ will be to the zero vector.
Also note that if
$\gamma^\perp_k<0$ then the solution to the subproblem (\ref{eq:q_perp2})
will always lie on the boundary, and thus, the actual value of
$\gamma^\perp_k$ becomes irrelevant.  Moreover,
for values $\gamma^\perp_k<0$, $\widehat{B}_k$ is not guaranteed to 
be positive definite. 
For these reasons, we suggest
sufficiently large and positive values for $\gamma^\perp_k$
related to the 
curvature information gathered in $\gamma_1, \ldots, \gamma_k$.
  Specific
choices for $\gamma^\perp_k$ are presented in the numerical results
section.

\subsection{Implementation details} \label{sec-algorithm}
In this section, we describe how we incorporate the dense initialization
within the existing
{\small LMTR} algorithm~\cite{BurdakovLMTR16}.  
At the beginning of each iteration, the 
{\small LMTR} algorithm with dense initialization
checks if the
unconstrained minimizer 
(also known as the \emph{full quasi-Newton trial step}), 
\begin{equation}\label{eqn-pstar2}
p_u^* = - \hat{B}_k^{-1} g_k
\end{equation}
lies inside the trust region defined by the two-norm. 
Because our
proposed method uses a dense initialization, the
so-called ``two-loop recursion'' [6] is not applicable for computing
the unconstrained minimizer $p_u^*$ in (\ref{eqn-pstar2}).  
However, products with
$\hat{B}_k^{-1}$ can be performed using the compact representation
without involving a partial 
eigendecomposition. %, i.e.,
Specifically, if   $V_k = \left[S_k \ Y_k\right]$
with Cholesky factorization $V_k^TV_k = R_k^TR_k$, then
%\begin{equation}\label{eqn-32}
%\hat{B}_k^{-1} = \frac{1}{\gamma_k^{\perp}}I +
%V_k \hat{M}_k V_k^T,
%\end{equation}
%where $V_k = \left[S_k \ Y_k\right]$,
%$$
%\hat{M}_k =
%\left[
%\begin{matrix}
%T_k^{-T}(D_k + \gamma_k^{-1}Y_k^TY_k)T_k^{-1} & -\gamma_k^{-1}T_k^{-T}\\
%-\gamma_k^{-1}T_k^{-1} & 0_m
%\end{matrix}
%\right] + \alpha_k \left(V_k^T V_k\right)^{-1},
%$$
\begin{equation}\label{eqn-32}
\hat{B}_k^{-1} = \frac{1}{\gamma_k^{\perp}}I +
V_k \hat{M}_k V_k^T,
\end{equation}
where 
$$
\hat{M}_k =
\left[
\begin{matrix}
T_k^{-T}(D_k + \gamma_k^{-1}Y_k^TY_k)T_k^{-1} & -\gamma_k^{-1}T_k^{-T}\\
-\gamma_k^{-1}T_k^{-1} & 0_m
\end{matrix}
\right] + \alpha_k R_k^{-1} R_k^{-T},
$$
$\displaystyle \alpha_k = \left(\frac{1}{\gamma_k}
- \frac{1}{\gamma_k^{\perp}}\right)$, $T_k$ is the upper triangular
part of the matrix $S_k^TY_k$, and $D_k$ is its diagonal.
Thus, the inequality 
\begin{equation}\label{eqn-unconstrainedmin}
    \|p_u^*\|_2 \le \Delta_k
\end{equation}
is easily verified without explicitly forming $p_u^*$ 
using the identity
\begin{equation}\label{eqn-pustar}
\|p_u^*\|_2^2 = g_k^T \hat{B}_k^{-2} g_k = 
\gamma_k^{-2}\|g_k\|^2 + 2\gamma_k^{-1} u_k^T \hat{M}_k u_k + u_k^T \hat{M}_k (R_k^T R_k) \hat{M}_k u_k.
\end{equation}
Here, as in the LMTR algorithm, the vector $u_k = V_k^T g_k$ 
and $\| g_k \|^2$ can be computed efficiently at each iteration
(see \cite{BurdakovLMTR16} for details).
%when updating the matrix $V_k^TV_k$. 
Thus, the computational cost of $\|p_u^*\|_2$ is low because 
\eqref{eqn-pustar} involves linear algebra operations in a small $2m$-dimensional space,
the most expensive of which are related to solving triangular systems with $T_k$ and $R_k$.
These operations grow in proportion to $m^2$ while the number of operations in 
\eqref{eqn-pstar2}-\eqref{eqn-32} grows in proportion to $mn$. Thus, the 
computational complexity ratio between using \eqref{eqn-pustar}  and
\eqref{eqn-pstar2}-\eqref{eqn-32} is  $m^2/(nm) = m/n \ll 1$ since we assume that $m \ll n$.
%the matrices $V_k^T V_k$ and $\hat{M}_k$ are small in size.  
%\je{In particular, computing $\| p_u^* \|_2^2$ in (34) only requires $n + 8m^2+4m+4$ multiplications 
%and $n+8m^2-3$ additions.  In contrast, forming $p_u^*$ explicitly in (31) using (32) requires
%$(2m+1)n+4m^2$ multiplications and $2mn+4m^2-2m$ additions.  Since $n \gg m$,
%the cost of using (34) is about $1/(2m+1)$ the cost of forming $p_u^*$ in (31).}
The norm
equivalence for the shape-changing infinity norm studied
in~\cite{BurdakovLMTR16} guarantees that (\ref{eqn-unconstrainedmin})
implies that the inequality $\|p_u^*\|_{P,\infty} \le \Delta_k$ is
satisfied; in this case, $p_u^*$ is the exact solution of the
trust-region subproblem defined by the shape-changing infinity
norm.

If (\ref{eqn-unconstrainedmin}) holds, the algorithm computes
$p_u^*$ for generating the trial point $x_k + p_u^*$.   It can be
easily seen that the cost of computing $p_u^*$ is $4mn$ operations,
i.e. it is the same as for computing search direction in the line
search L-BFGS algorithm [6].

On the other hand, if (\ref{eqn-unconstrainedmin})
does not hold, then for producing a trial point, the partial
eigendecomposition is computed, and the trust-region subproblem is
decoupled and solved exactly as described in Section \ref{subsec-decoupled}.

\subsection{The algorithm and its properties}
In Algorithm~\ref{alg}, we present a basic trust-region method that
uses the proposed dense initialization. In this setting, we consider
the computational cost of the proposed method, and we prove global
convergence of the overall trust-region method.  Since $P$ may change
every iteration, the corresponding norm $\|\cdot\|_{P,\infty}$ may
change each iteration.  Note that initially there are no stored
quasi-Newton pairs $\{s_j,y_j\}$.  In this case, we assume $P_{\perp}
= I_n$ and $P_{\parallel}$ does not exist, i.e., $\hat{B}_0=\gamma_0^\perp
I$.

\begin{algorithm}[htp]
\caption{An L-BFGS trust-region method with dense initialization} \label{alg}
\begin{algorithmic}[1]
\REQUIRE $x_0\in R^n$, \ $\Delta_0>0$, \ $\epsilon > 0$, \ $\gamma_0^{\perp}>0$\
, \ $0 \leq \tau_1 < \tau_2< 0.5 < \tau_3<1$, \\
$0<\eta_1<\eta_2\leq 0.5<\eta_3<1<\eta_4$, $0 < c_3 < 1 $
\STATE Compute $g_0$
\FOR{$k=0,1,2,\ldots$}
        \IF{$\|g_k\|\leq\epsilon$}
        \RETURN
        \ENDIF
        \STATE %\je{Compute the unconstrained minimizer $p_u^*$ }
        Compute $\| p_u^* \|_2$ using \eqref{eqn-pustar}
        \IF{ $\|p_u^*\|_2 > \Delta_k$} 
              \STATE Compute $p^*$ for $\hat{B}_k$ using (\ref{eqn-Roummel}), where $\widehat{\beta}$ is computed using (\ref{eq:subsolnbeta2}) and $\vec{v}^*_{\parallel}$ as in (\ref{eqn:solution_vparallel})
        \ELSE
           \STATE Compute $p_u^*$ using \eqref{eqn-pstar2}-\eqref{eqn-32} and set $p^*\gets p_u^*$
        \ENDIF
        \STATE Compute the ratio $\rho_k = \frac{f(x_k+p^*)-f(x_k)}{Q(p^*)}$
        \IF{$\rho_k {\geq \tau_1}$}
                \STATE $x_{k+1}=x_k+p^*$
                \STATE Compute $g_{k+1}$, $s_k$, $y_k$, $\gamma_{k+1}$ and $\gamma_{k+1}^{\perp}$
        \STATE Choose at most $m$ pairs $\{s_j, y_j\}$ such that $ s_j^Ty_j > c_3 \| s_j \| \| y_j \| $
        \STATE Compute $\Psi_{k+1}, R^{-1}, M_{k+1}, W, \hat{\Lambda}$ and $\Lambda$ 
as described in Section~\ref{sec:background}

        \ELSE
                \STATE $x_{k+1} = x_k$
        \ENDIF
        \IF{$\rho_k < \tau_2$}
                \STATE $\Delta_{k+1} = \min \left({\eta_1}\Delta_k, {\eta_2}\|s_k\|_{P,\infty} \right)$
        \ELSE
                \IF{$\rho_k \geq \tau_3$ \AND $\|s_k\|_{P,\infty}\geq {\eta_3} \Delta_k$\
}
                        \STATE $\Delta_{k+1} = {\eta_4} \Delta_k$
                \ELSE
                        \STATE{$\Delta_{k+1}=\Delta_k$}
                \ENDIF
        \ENDIF
\ENDFOR
\end{algorithmic}
\end{algorithm}

The only difference between Algorithm~\ref{alg} and the %proposed
{\small LMTR} algorithm in~\cite{BurdakovLMTR16} is the initialization
matrix.  Computationally speaking, the use of a dense initialization
in lieu of a diagonal initialization plays out only in the computation
of $p^*$ by (\ref{eqn-finalsolution}).  However, there is no
computational cost difference: The cost of computing the value for
$\beta$ using (\ref{eq:subsolnbeta2}) in Algorithm~\ref{alg} instead
of (\ref{eq:subsolnbeta}) in the {\small LMTR} algorithm is the same.
Thus, the dominant cost per iteration for both Algorithm~\ref{alg} and
the {\small LMTR} algorithm is $4mn$ operations
(see~\cite{BurdakovLMTR16} for details).  Note that this is the same
cost-per-iteration as the line search \LBFGS{}
algorithm~\cite{ByrNS94}.

\medskip

In the next result, we provide a global convergence result for
Algorithm~\ref{alg}.  This result is based on the convergence analysis
presented in~\cite{BurdakovLMTR16}.

\begin{theorem}\label{th_conv}
Let  $f:R^n\rightarrow R$ be twice-continuously differentiable and bounded below on $R^n$. Suppose that there exists a scalar $c_1>0$ such that 
\begin{equation}\label{bound_hessian}
\|\nabla^2f(x)\|\leq c_1, \ \forall x \in R^n.
\end{equation}
Furthermore, suppose for $\hat{B}_0$ defined by (\ref{eq:denseB0}), that there exists a positive scalar $c_2$ such that
\begin{equation}\label{bound_gamma}
\gamma_k , \gamma_k^{\perp} \in (0,c_2], \ \forall k\ge 0,
\end{equation}
and there exists a scalar $c_3 \in (0,1)$ such that the inequality
\begin{equation}\label{bound_sy}
s_j^T y_j >   c_3 \|s_j\| \|y_j\|
\end{equation} 
holds for each quasi-Newton pair $\{s_j, y_j\}$. 
%$\hat{B}_k$ by formula Algorithm~\ref{alg}(L213').
Then, if the
stopping criteria is suppressed, the infinite sequence $\{x_k\}$
generated by Algorithm~\ref{alg}
satisfies \begin{equation}\label{conv}
\lim _{k\rightarrow \infty} \|\nabla f(x_k)\| = 0.
\end{equation}
\end{theorem}
\begin{proof}
From \eqref{bound_gamma}, we have 
$\|\hat{B}_0\| \le c_2,$ which
holds for each $k \ge 0$.  Then,
by~\cite[Lemma~3]{BurdakovLMTR16}, there exists $c_4 > 0$ such that 
$$
\|\hat{B}_k\| \le c_4.
$$
Then, \eqref{conv} follows from~\cite[Theorem~1]{BurdakovLMTR16}. $\square$
\end{proof}

\bigskip

In the following section, we consider 
$\gamma_k^{\perp}$ parameterized by two scalars, $c$ and $\lambda$:
\begin{equation}\label{eqn-gammaperp1}
\gamma_k^{\perp}(c,\lambda) = \lambda c \gamma_k^{\max} + (1 - \lambda)\gamma_k,
\end{equation}
where $c \ge 1, \lambda \in [0,1]$, and
$$
 \gamma_k^{\text{max}}\defined \underset{1 \le i \le k}{ \text{ max } \gamma_i },
 $$
where  $\gamma_k$ is taken to be
the conventional initialization given by (\ref{eqn-B0-usual}).  
(This choice for $\gamma_k^\perp$ will be further discussed
in Section \ref{sec:numexp}.)  We now show that Algorithm \ref{alg}
converges for these choices of $\gamma_k^\perp$.  Assuming that
(\ref{bound_hessian}) and (\ref{bound_sy}) hold, it remains to show that
(\ref{bound_gamma}) holds for these choices of $\gamma_k^\perp$.  To see
that (\ref{bound_gamma}) holds, notice that in this case, $$\gamma_k
= \frac{y_k^Ty_k}{s_k^Ty_k}\le \frac{y_k^Ty_k}{c_3\|s_k\|\|y_k\|}
\le \frac{\|y_k\|}{c_3\|s_k\|}.$$  Substituting in for the definitions of $y_k$ and $s_k$ yields that
$$
\gamma_k \le \frac{\|\nabla f(x_{k+1})-\nabla f(x_k)\|}{c_3\|x_{k+1}-x_k\|},$$
implying that (\ref{bound_gamma}) holds.  Thus, Algorithm~\ref{alg} converges
for these choices for $\gamma_k^\perp$.

%% file: 4-NumericalExperiments.tex
%We perform numerical experiments on 65 large-scale ($1000\le n \le 10000$)
%{\small CUTE}st~\cite{GouOT03} test problems, made up of all the test problems
%in~\cite{BurdakovLMTR16} plus an additional three ({\small FMINSURF},
%{\small PENALTY2}, and {\small TESTQUAD}~\cite{GouOT03}) since at least one of the methods
%in the experiments detailed below converged on one of these three
%problems.  
We performed numerical experiments using a Dell Precision T1700 machine with an Intel i5-4590 {\small CPU} at $3.30${\small GH}z X4 and 8{\small GB RAM} using {\small MATLAB} 2014a.
The test set consisted of 
 65 large-scale ($1000\le n \le 10000$)
{\small CUTE}st~\cite{GouOT03} test problems, made up of all the test problems
in~\cite{BurdakovLMTR16} plus an additional three ({\small FMINSURF},
{\small PENALTY2}, and {\small TESTQUAD}~\cite{GouOT03}) since at least one of the methods
in the experiments detailed below converged on one of these three
problems.  
The same trust-region method and default parameters as
in~\cite[Algorithm 1]{BurdakovLMTR16} were used for the outer iteration.
At most five quasi-Newton pairs $\{ s_k, y_k \}$ were stored,
i.e., $m = 5$.
The relative stopping criterion was 
$$ \left\| \vec{g}_k \right\|_2 \le \epsilon \max \left( 1,
  \left\| \vec{x}_k \right\|_2 \right), $$ 
with $\epsilon=10^{-10}$.  The
initial step, $p_0$, was determined by a backtracking line-search along the
normalized steepest descent direction.  
To compute the partial eigendecomposition of $B_k$, we used the {\small QR} factorization
instead of the {\small SVD} because the  {\small QR} version outperformed the {\small SVD} 
version in numerical experiments not presented here.
The rank of $ \vec{\Psi}_k $ was
estimated by the number of positive diagonal elements in the diagonal
matrix of the $\text{LDL}^{\text{T}}$ decomposition (or eigendecomposition
of $ \vec{\Psi}^T_k \vec{\Psi}_k $) that are larger than the threshold $
\epsilon_r = (10^{-7})^2 $. (Note that the columns of $\Psi_k$
are normalized.). We used the value $c_3 = 10^{-8}$ in \eqref{bound_sy}
for testing whether to accept a new quasi-Newton pair.

We provide performance profiles (see \cite{DolanMore02}) for the number of
iterations (\texttt{iter}) where the trust-region step is 
accepted
and the average time (\texttt{time}) for each
solver on the test set of problems.  The performance metric, $ \rho $, for
the 65 problems is defined by
	\begin{equation*}
		\rho_s(\tau) = \frac{\text{card}\left\{ p : \pi_{p,s} \le \tau \right\}}{65} \quad \text{and} \quad \pi_{p,s} = \frac{t_{p,s}}{ \underset{1\le i \le S}{\text{ min } t_{p,i}} },
	\end{equation*} 
	where $ t_{p,s}$ is the ``output'' (i.e., time or iterations) of
        ``solver'' $s$ on problem $p$. Here $ S $ denotes the total number of solvers for a given comparison. This metric measures
        the proportion of how close a given solver is to the best
        result.  We observe as in \cite{BurdakovLMTR16} that the first runs
        significantly differ in time from the remaining runs, and thus, we
        ran each algorithm ten times on each problem, reporting the average
        of the final eight runs.

\bigskip

In this section, we present the following six types of experiments involving {\small LMTR}:
\newline

\begin{enumerate}
\item A comparison of results for different values of $\gamma_k^{\perp}(c,\lambda)$. 
\item Two versions of computing the full quasi-Newton trial step %(see Section 3.4)
 are compared. One version uses the dense initialization to
    compute $p_u^*$ as described in Section 3.4 (see (\ref{eqn-pstar2})); the other uses the conventional
    initialization, i.e., $p_u^*$ is computed as
    $p_u^*=B_k^{-1}g_k$. 
    %In the both cases, the dense initialization is
    %used for computing trial steps obtained from explicitly solving
    %the trust-region subproblem (Section 3.2) when the full
    %quasi-Newton trial step is not accepted.
    When the full quasi-Newton trial step is not accepted in any of the versions, the dense initialization is used for computing trial step by explicitly solving the trust-region subproblem (Section 3.2).
%\item A comparison of alternative ways of computing the partial eigendecomposition (Section 2.2), namely, those based on {\small QR} and {\small SVD} factorizations.
\item A comparison of {\small LMTR} together with a
dense initialization and the line search {\small L-BFGS} method
with the conventional initialization.
\item A comparison of {\small LMTR} 
with a
dense initialization and 
{\small L-BFGS-TR}~\cite{BurdakovLMTR16},
which computes a scaled quasi-Newton direction that lies inside a trust
region.  This method can be viewed as a hybrid line search and trust-region
algorithm.
\item A comparison of the dense and conventional initializations.
\end{enumerate}

\medskip 

In the experiments below, the dense initial matrix $\widehat{B}_0$
corresponding to $\gamma_k^{\perp}(c,\lambda)$ given in \eqref{eqn-gammaperp1}
will be denoted by
$$
	\widehat{B}_0(c,\lambda) 
	\defined
	\gamma_k P_{\parallel}P_{\parallel}^T + 
	\gamma_k^{\perp}(c,\lambda) P_{\perp}P_{\perp}^T.
$$
Using this notation, the conventional initialization $B_0(\gamma_k)$
can be written as $\widehat{B}_0(1,0)$.

\medskip

\noindent
{\bf Experiment 1.}
In this experiment, we consider various scalings of a proposed $ \gamma^{\perp}_{k} $ using
{\small LMTR}.  
As argued in Section \ref{subsec:gammaperp},  it is reasonable
to choose $\gamma_k^\perp$ to be large and positive; in particular, 
$\gamma_k^\perp\ge\gamma_k$.  Thus, we 
consider the parametrized family of choices $\gamma_k^{\perp} \defined \gamma_k^{\perp}(c,\lambda)$
given in \eqref{eqn-gammaperp1}.
	These choices correspond
        to conservative strategies for computing steps in the space spanned
        by $ \vec{P}_{\perp} $ (see the discussion in Section \ref{subsec:gammaperp}).
Moreover, these can also be viewed as conservative strategies since
the trial step computed using $B_0$ will always be larger in Euclidean
norm than the trial step computed using $\widehat{B}_0$ using (\ref{eqn-gammaperp1}).
To see this, note that in the parallel subspace the solutions will
be identical using both initializations since the solution $v_\parallel^*$ does
not depend on $\gamma_k^\perp$ (see (\ref{eqn:solution_vparallel})); in contrast,
in the orthogonal subspace, $\|v_\perp^*\|$ inversely depends on $\gamma_k^\perp$
(see (\ref{eq:subsolnperp2}) and (\ref{eq:subsolnbeta2})).

\medskip

We report results using different
values of $c$ and $\lambda$ for $\gamma^\perp_k(c,\lambda)$ 
on two sets of tests.  
On the first set of tests, 
the dense initialization was used for the entire {\small LMTR} algoirthm.
However, for the second set of tests, 
the dense initialization was not used for the computation
of the unconstrained minimizer $p_u^*$;  
that is,  {\small LMTR} was run
using $B_k$ (initialized with $B_0=\gamma_k I$ where $\gamma_k$ is given in
(\ref{eqn-B0-usual})) for the computation of the unconstrained minimizer
$p_u^*=-B_k^{-1}g_k$.  However, if the unconstrained minimizer was not
taken to be the approximate solution of the subproblem, $\widehat{B}_k$
with the dense initialization was used for computing the constrained minimizer
with respect to the shape-changing norm (see line 8 in Algorithm 1)
%the
%shape-changing component 
%of the algorithm 
with
$\gamma_k^{\perp}$ defined as in \eqref{eqn-gammaperp1}.
The values of $c$ and $\lambda$ chosen for Experiment 1 are found in Table \ref{table:gammaperp-JE}.
(See Section~\ref{sec-algorithm} for details on the {\small LMTR} algorithm.)

\begin{table} \label{table:gammaperp-JE}
\caption{Values for $\gamma_k^{\perp}$ used in Experiment 1.}
\renewcommand{\arraystretch}{1.2}
\begin{tabular}{ccl}
\hline
\multicolumn{2}{c}{Parameters} \\
\cline{1-2}
$c$ & $\lambda$ & $\gamma_k^{\perp}$  \\  
\hline
$1$ & $1$ & $\gamma_k^{\max}$  \\
$2$ & $1$ & $2\gamma_k^{\max}$ \\
 $1$ & $\frac{1}{2}$ & $\frac{1}{2} \gamma_k^{\max} +  \frac{1}{2}\gamma_k$ \\
$1$ & $\frac{1}{4}$ & $\frac{1}{4}\gamma_k^{\max} + \frac{3}{4} \gamma_k$  \\ 
\hline
\end{tabular}
\end{table}

Figure~\ref{fig:1B} displays the
performance profiles 
using  the chosen values of $c$ and $\lambda$ to define $\gamma_k^{\perp}$ 
in the case when the
dense initialization was used for both the computation
of the unconstrained minimizer $p_u^*$ (line 10 of Algorithm 1) as well as for the
constrained minimizer with respect to the shape-changing norm (line 8 of Algorithm 1),
%shape-changing component 
%of the algorithm, 
which is denoted in the legend of plots in
Figure~\ref{fig:1B} by the use of an asterisk $ (*)$.
The results of Figure~\ref{fig:1B} suggest the choice
of $c=1$ and $\lambda=\frac{1}{2}$ outperform the other
chosen combinations for $c$ and $\lambda$.
In experiments not reported here, 
larger values of $c$ did not appear to improve performance; for $c<1$,
performance deteriorated.
Moreover, other choices for 
$\lambda$, such as $ \lambda=\frac{3}{4}$, did not improve results
beyond the choice of $\lambda=\frac{1}{2}$.

	\begin{figure*}[h!]
				\begin{minipage}{0.48\textwidth}
					\includegraphics[width=\textwidth]{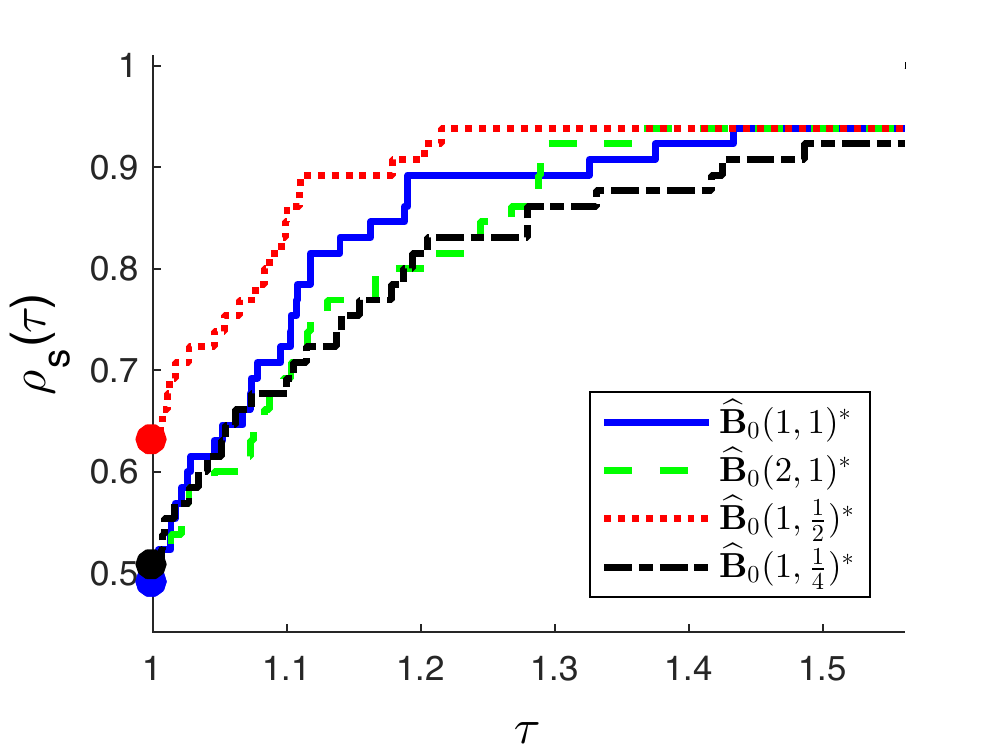}
				\end{minipage}
				\hfill
				\begin{minipage}{0.48\textwidth}
							\includegraphics[width=\textwidth]{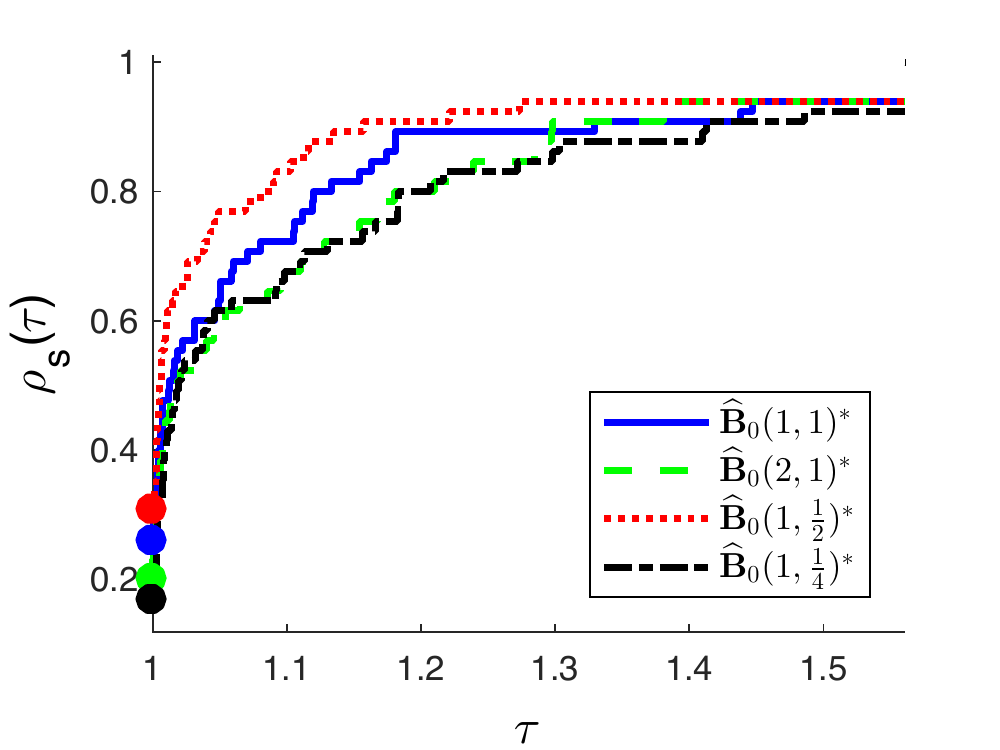}
				\end{minipage}
				\caption{
Performance profiles comparing \texttt{iter} (left) and \texttt{time} (right)  
for the different values of $\gamma_k^\perp$
given in Table~\ref{table:gammaperp-JE}. 
In the legend, $\widehat{B}_0(c,\lambda)$ denotes the results from 
using the dense initialization
with the given values for $c$ and $\lambda$ to define $\gamma_k^\perp$.
In this experiment, the dense initialization
was used for all aspects  
of the algorithm.}
			\label{fig:1B}       
	\end{figure*}

Figure~\ref{fig:1A} reports the 
performance profiles for 
using the chosen values of $c$ and $\lambda$ to define $\gamma_k^{\perp}$ 
in the case when 
the dense initialization 
was only used for the
computation of the constrained minimizer (line 8 of Algorithm 1)
%shape-changing component 
--denoted in the legend of plots in
Figure~\ref{fig:1A} by the absence of an asterisk $ (*)$.
In this test, the combination of
$c=1$ and $\lambda=1$  as well as 
$c=1$ and $\lambda=\frac{1}{2}$
appear to slightly outperform
the other two choices for  $\gamma^\perp_k$ in terms of both
then number of iterations and the total computational time.
Based on the results in Figure~\ref{fig:1A}, 
we do not see a reason to prefer either
combination $c=1$ and $\lambda=1$ or
$c=1$ and $\lambda=\frac{1}{2}$
over the other.

Note that for the {\small CUTE}st problems, the full quasi-Newton trial step is accepted
as the solution to the subproblem
on the overwhelming
majority of problems. Thus, if the scaling $\gamma_k^\perp$ is
used only when the full trial step is rejected, it has less of an affect
on the overall performance of the algorithm; i.e., the algorithm is less
sensitive to the choice of $\gamma_k^\perp$.  For this reason, it is not
surprising that the performance profiles in Figure~\ref{fig:1A}
for the different values
of $\gamma_k^\perp$ are more indistinguishable than those in
Figure~\ref{fig:1B}. 

Finally, similar to the results in the case when the dense initialization
was used for the entire algorithm (Figure~\ref{fig:1B}), other values of
$c$ and $\lambda$ did not significantly improve the performance
provided by $c=1$ and $\lambda=\frac{1}{2}$.

	\begin{figure*}[h!]
				\begin{minipage}{0.48\textwidth}
					\includegraphics[width=\textwidth]{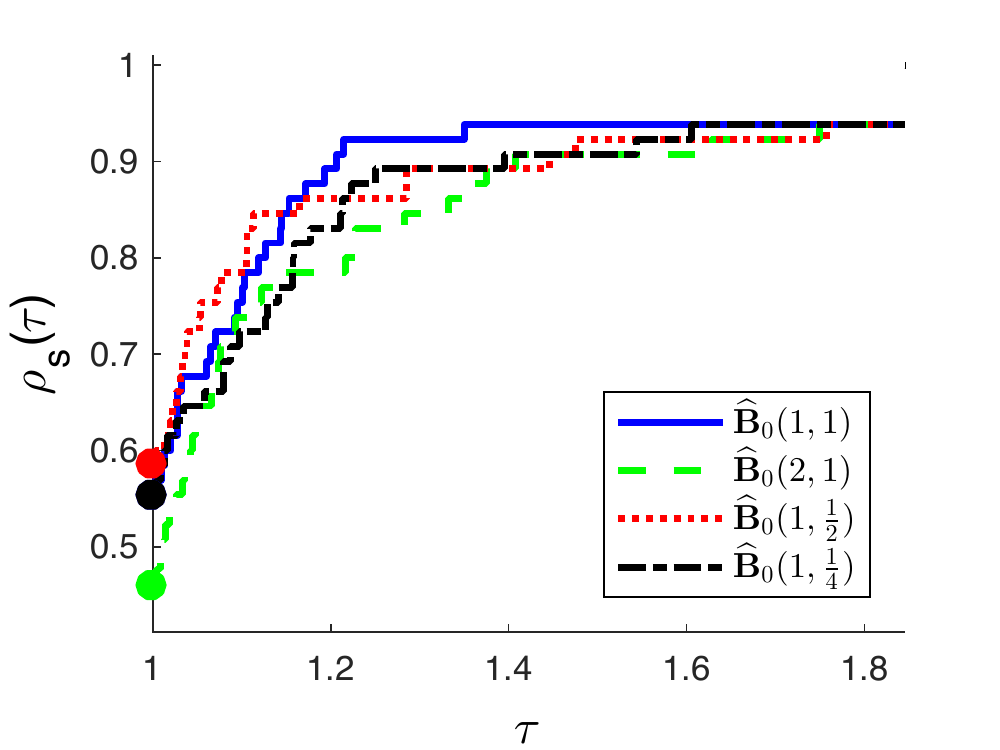}
				\end{minipage}
				\hfill
				\begin{minipage}{0.48\textwidth}
							\includegraphics[width=\textwidth]{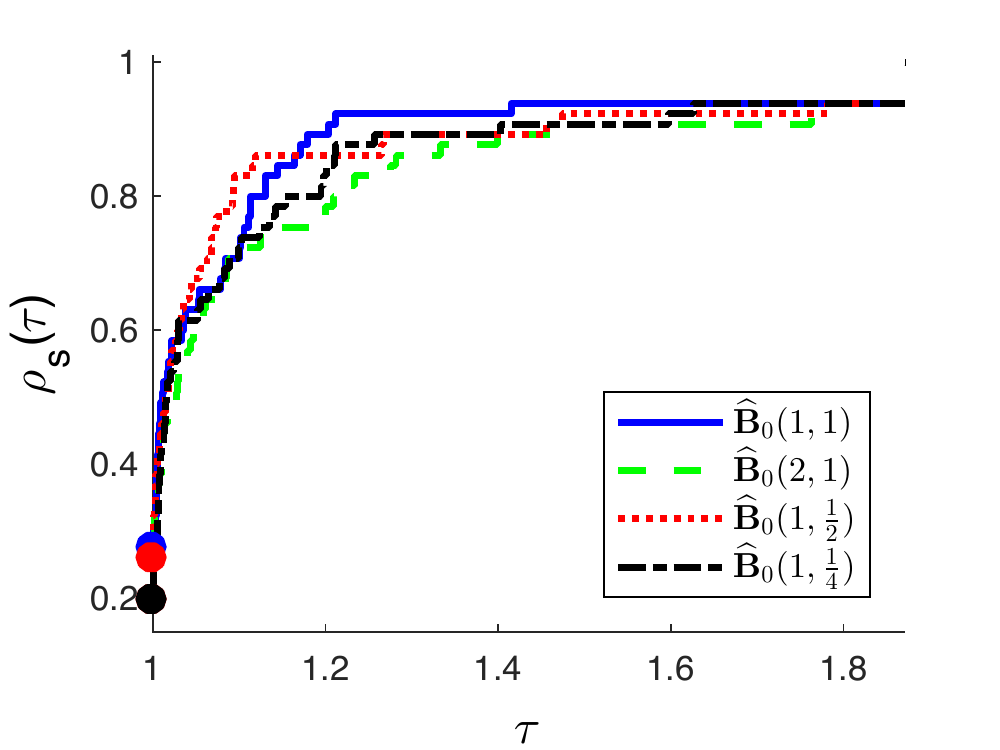}
				\end{minipage}
				\caption{
Performance profiles comparing \texttt{iter} (left) and \texttt{time} (right)  
for the different values of $\gamma_k^\perp$
given in Table~\ref{table:gammaperp-JE}. 
In the legend, $\widehat{B}_0(c,\lambda)$ denotes the results from 
using the dense initialization
with the given values for $c$ and $\lambda$ to define $\gamma_k^\perp$.
 In this experiment, the dense initialization
was only used for the computation of the constrained minimizer (line 8 of Algorithm 1).}
			\label{fig:1A}       
	\end{figure*}

\bigskip
\noindent
{\bf Experiment 2.} This experiment was designed to test whether
it is advantageous to use the dense initialization for all aspects of the
{\small LMTR} algorithm or just for the computation of the constrained minimizer (line 8 of Algorithm 1).
For any given trust-region subproblem, using the dense initialization for computing the unconstrained minimizer is computationally
more expensive than using a diagonal initialization; however, it is possible
that extra computational time associated with using the dense initialization for all
aspects of the {\small LMTR} algorithm may yield
a more overall efficient solver.
For these tests, we compare the top performer 
in the case when the dense initialization
is used for all aspects of {\small LMTR}, 
i.e., $(\gamma_k^{\perp}(1,\frac{1}{2}))$, to one of the top performers in
  the case when the dense initialization is used only for the
computation of the constrained minimizer (line 8 of Algorithm 1),
  i.e., $(\gamma_k^{\perp}(1,1))$.

	\begin{figure*}[h!]
				\begin{minipage}{0.48\textwidth}
					\includegraphics[width=\textwidth]{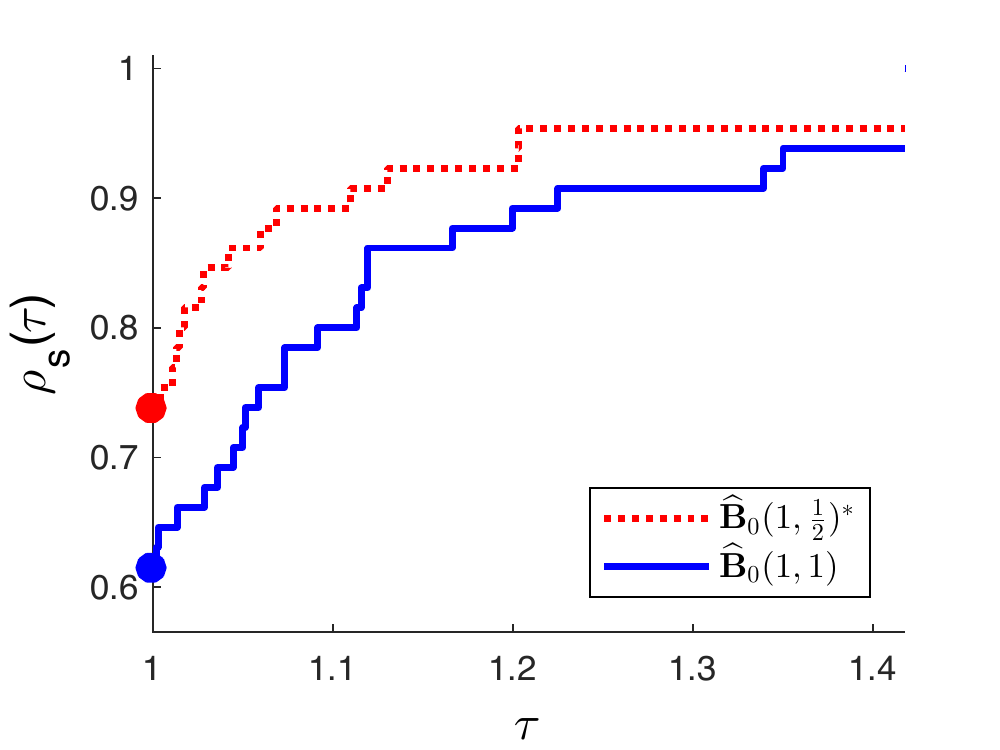}
				\end{minipage}
				\hfill
				\begin{minipage}{0.48\textwidth}
							\includegraphics[width=\textwidth]{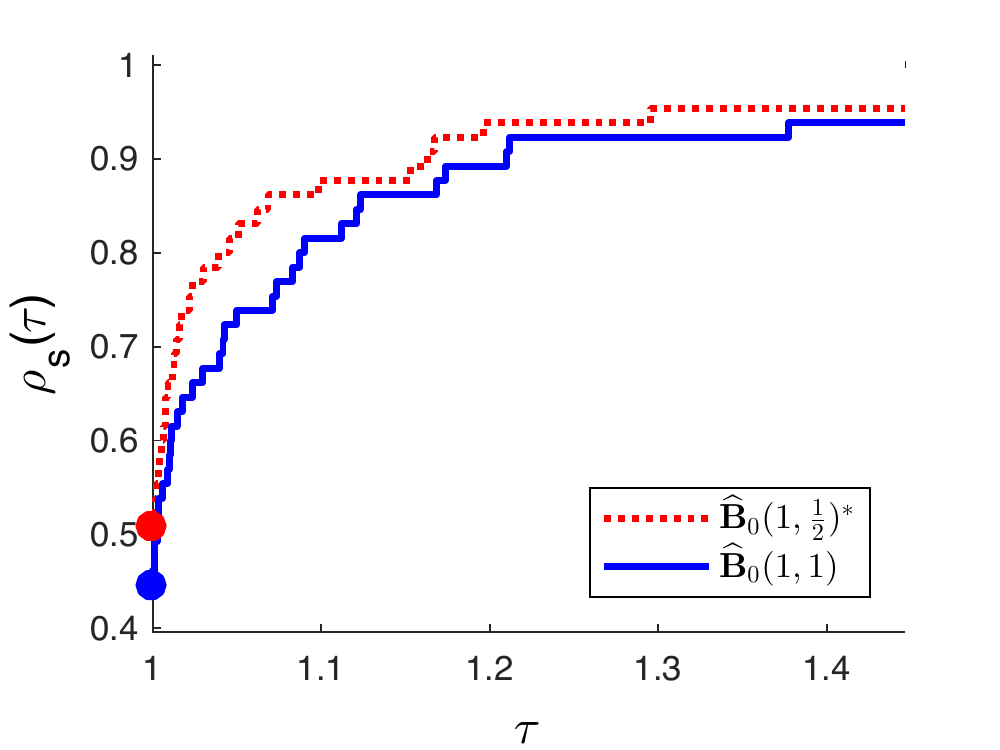}
				\end{minipage}
				\caption{
Performance profiles of \texttt{iter} (left) and
\texttt{time} (right) for Experiment 2.  
In the legend, 
the asterisk after
$\widehat{B}_0(1,\frac{1}{2})^*$
 signifies that the dense initialization 
was used for all aspects of the {\small LMTR} algorithm;
without the asterisk,
$\widehat{B}_0(1,1)$
signifies the test where the dense initialization 
is used only for the computation of the constrained minimizer (line 8 of Algorithm 1).
}
			\label{fig:2}       
	\end{figure*}

The performance profiles comparing
the results of this experiment
are presented in Figure \ref{fig:2}.   
        These results suggest that using the dense initialization
        with 
        $\gamma_k^{\perp}(1,\frac{1}{2})$
	for all aspects of
	the {\small LMTR} algorithm is more efficient than using dense
	initializations only for the computation of the constrained minimizer (line 8 of Algorithm 1).
	In other words, even though using dense initial matrices for the computation of the
        unconstrained minimizer imposes 
        an additional computational burden,
        it generates steps that expedite the convergence of the overall trust-region method.

\bigskip
\noindent
{\bf Experiment 3.}  In this experiment, we compare the performance of the
dense initialization
$\gamma_k^{\perp}(1,0.5)$
to that of the line-search {\small L-BFGS} algorithm.
 For this comparison, we used the publicly-available {\small MATLAB}                
wrapper~\cite{BeckerLbfgs} for the {\small FORTRAN} {\small L-BFGS-B} code     
developed by Nocedal et al.~\cite{ZhuByrdNocedal97}.  The initialization for {\small L-BFGS-B} is $B_0=\gamma_kI$ where             
$\gamma_k$ is given by (\ref{eqn-B0-usual}).           
To make the stopping criterion equivalent to that of {\small L-BFGS-B},
we modified the stopping criterion of our solver         
to~\cite{ZhuByrdNocedal97}:                                                   
      \begin{equation*}                                                       
               \left\| \vec{g}_k \right\|_{\infty} \le \epsilon.              
       \end{equation*}                                                         
The dense initialization was used for all aspects of {\small LMTR}.
         
         The performance profiles for this experiment is given in Figure~\ref{fig:exp5}.
         On this test set, the dense initialization outperforms {\small L-BFGS-B}
         in terms of both the number of iterations and the total computational time.

  \begin{figure*}[h!]
                        \begin{minipage}{0.48\textwidth}
                                \includegraphics[width=\textwidth]{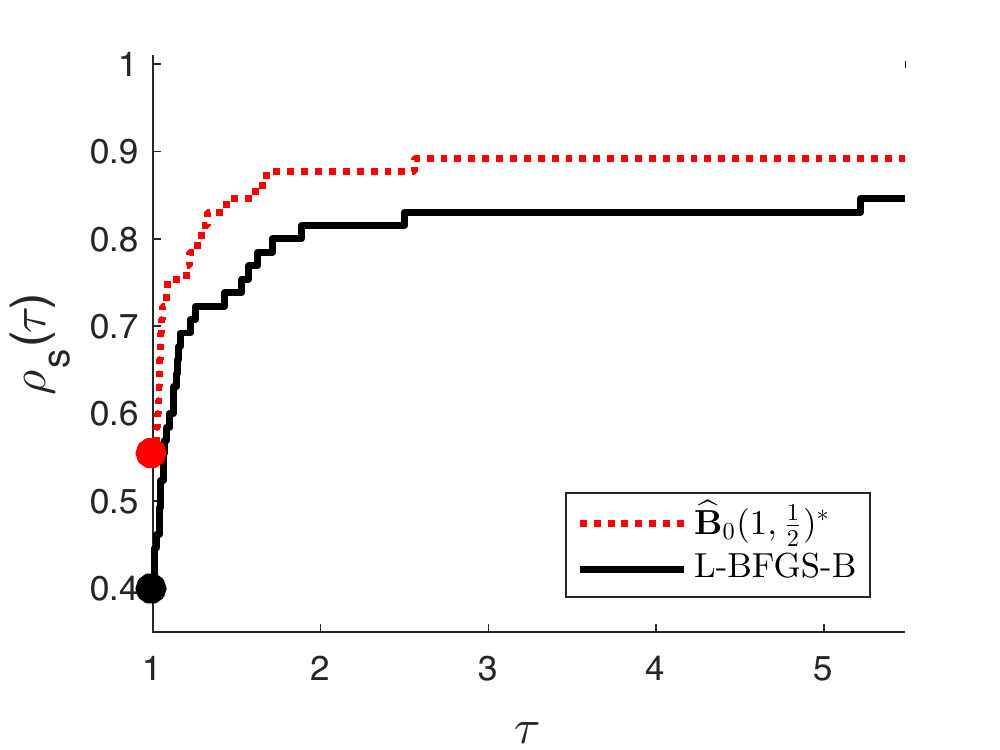}
                                
                        \end{minipage}
                        \hfill
                        \begin{minipage}{0.48\textwidth}
                                                \includegraphics[width=\textwidth]{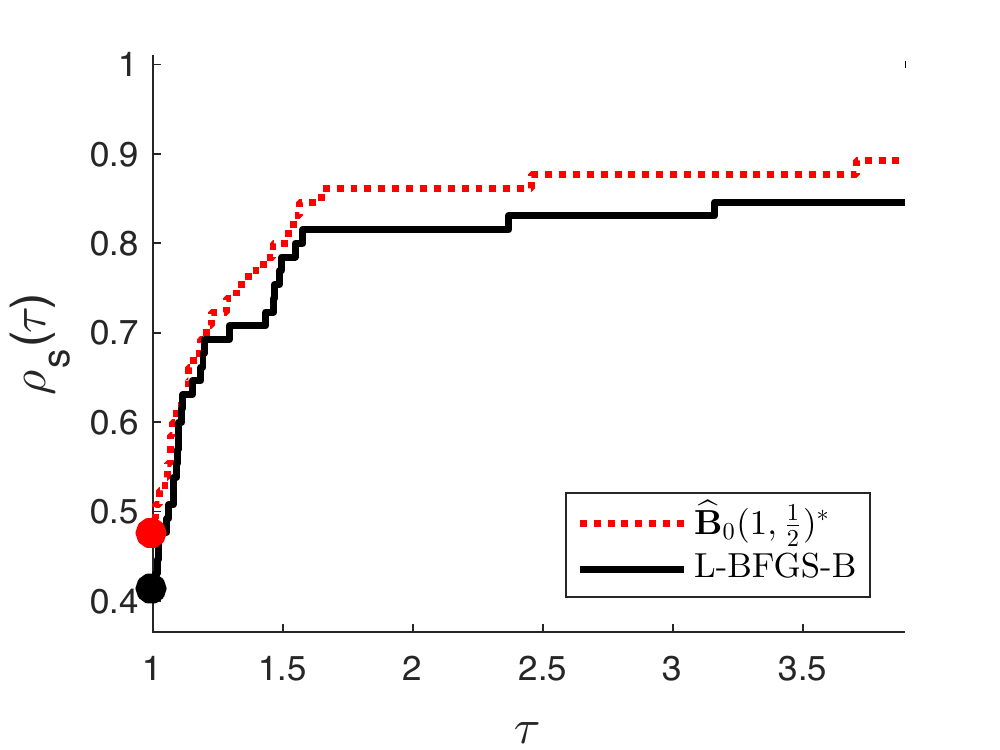}
                                                
                        \end{minipage}
                        \caption{Performance profiles of 
                        \texttt{iter} (left) and
\texttt{time} (right) for Experiment 3 comparing LMTR with the dense initialization
with
$\gamma_k^{\perp}(1,\frac{1}{2})$
 to {\small L-BFGS-B}.}
                \label{fig:exp5}       
                \end{figure*}

\bigskip
\noindent
{\bf Experiment 4.} In this experiment, we compare {\small LMTR} with
a dense initialization to {\small L-BFGS-TR}~\cite{BurdakovLMTR16},
which computes 
an \LBFGS{} trial step whose length is bounded by a
trust-region radius.  This method can be viewed as a hybrid \LBFGS{}
line search and trust-region algorithm 
because it uses a standard trust-region framework (as {\small LMTR}) but computes
a trial point by minimizing the quadratic model in the trust region
along the \LBFGS{} direction.
In~\cite{BurdakovLMTR16}, it was determined that this algorithm
outperforms two other versions of \LBFGS{} that use a Wolfe line search.
(For further details, see~\cite{BurdakovLMTR16}.)

  \begin{figure*}[h!]
                        \begin{minipage}{0.48\textwidth}
                                \includegraphics[width=\textwidth]{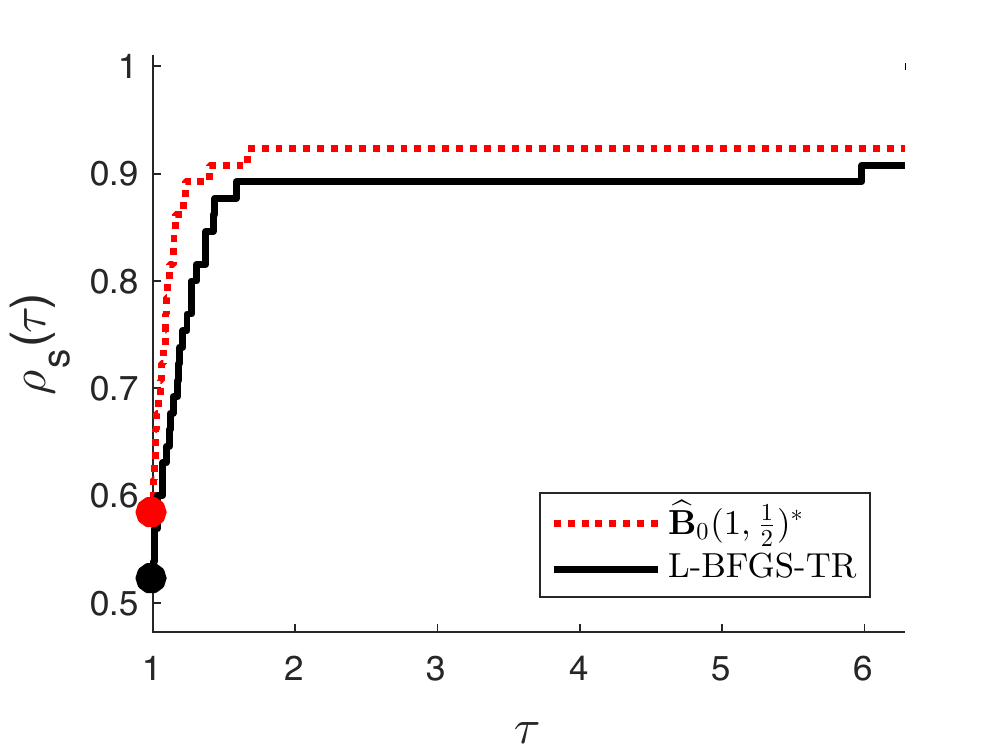}
                                
                        \end{minipage}
                        \hfill
                        \begin{minipage}{0.48\textwidth}
                                                \includegraphics[width=\textwidth]{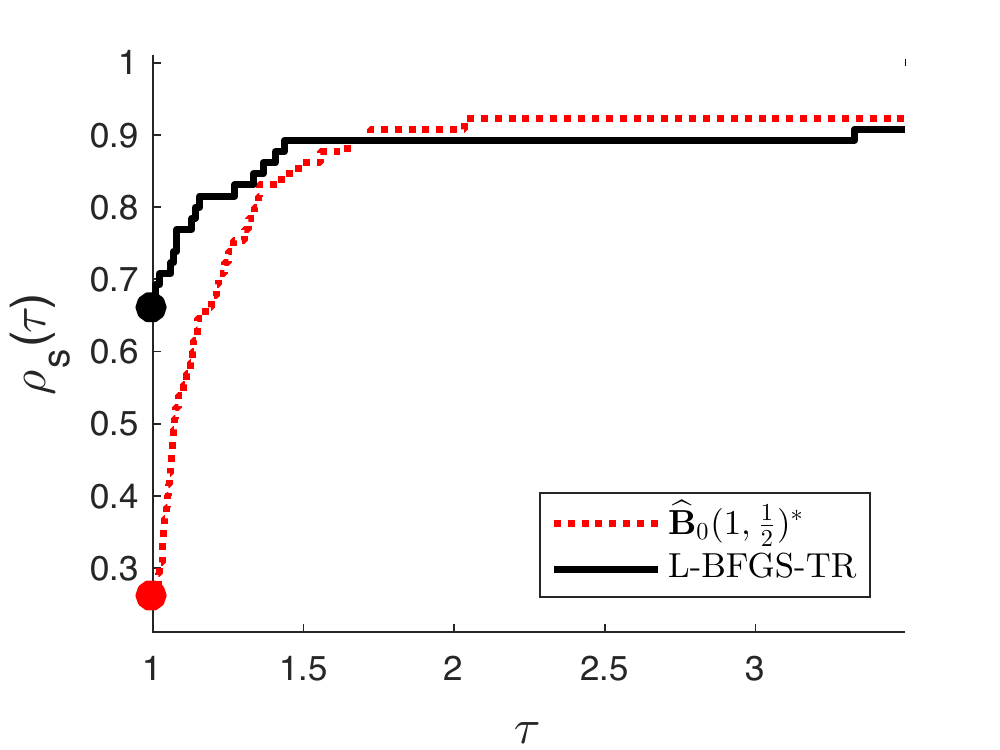}
                                
                        \end{minipage}
                        \caption{Performance profiles of 
                        \texttt{iter} (left) and
\texttt{time} (right) for Experiment 4
                        comparing  LMTR with the dense initialization
with 
$\gamma_k^{\perp}(1,\frac{1}{2})$
 to L-BFGS-TR.}
     
                \label{fig:exp6}       
                \end{figure*}
                
Figure~\ref{fig:exp6} displays the performance profiles associated with this experiment on
the entire set of test problems.  For this experiment, the dense initialization 
with $\gamma_k^{\perp}(1,\frac{1}{2})$ was used
in all aspects of the {\small LMTR} algorithm.
In terms of total number of iterations, {\small LMTR} with the dense initialization
outperformed {\small L-BFGS-TR}; however, {\small L-BFGS-TR} appears to have
outperformed {\small LMTR} with the dense initialization in computational time.

Figure~\ref{fig:exp6} (left) indicates that the quality of the trial
  points produced by solving the trust-region subproblem exactly using
  {\small LMTR} with the dense initialization is generally better than in
  the case of the line search applied to the \LBFGS{} direction.
 However, Figure~\ref{fig:exp6} (right) shows that {\small LMTR} with the dense
initialization requires more computational effort than {\small L-BFGS-TR}.
  For the
{\small CUTE}st set of test problems, {\small L-BFGS-TR} does not need to
perform a line search for the majority of iterations; 
that is, the full quasi-Newton trial step is accepted in a majority of the
iterations.  Therefore, we also compared the two algorithms on a subset of the
most difficult test problems--namely, those for which an \emph{active} line
search is needed to be performed by {\small L-BFGS-TR}.  To this end, we
select, as in~\cite{BurdakovLMTR16}, those of the {\small CUTE}st problems
in which the full \LBFGS{} (i.e., the step size of one)
was rejected in
at least 30\% of the iterations.  The number of problems in this subset
is 14.
The performance profiles associated with this reduced test set are in
Figure~\ref{fig:exp6-sel}.  On this smaller test set, {\small LMTR} outperforms
{\small L-BFGS-TR} both in terms of total number of iterations and computational time.

Finally, Figures~\ref{fig:exp6} and~\ref{fig:exp6-sel} suggest that when function
and gradient evaluations are expensive (e.g., simulation-based applications), {\small
LMTR} together with the dense initialization
is expected to be more efficient than {\small L-BFGS-TR}
since both on both test sets {\small LMTR} with
the dense initialization requires fewer overall iterations.  Moreover, Figure~\ref{fig:exp6-sel}
suggests that on problems where
the \LBFGS{} search direction 
often does not provide sufficient decrease of the objective function, {\small LMTR} with the 
dense initialization is expected to perform better.

 \begin{figure*}[h!]
                        \begin{minipage}{0.48\textwidth}
                                \includegraphics[width=\textwidth]{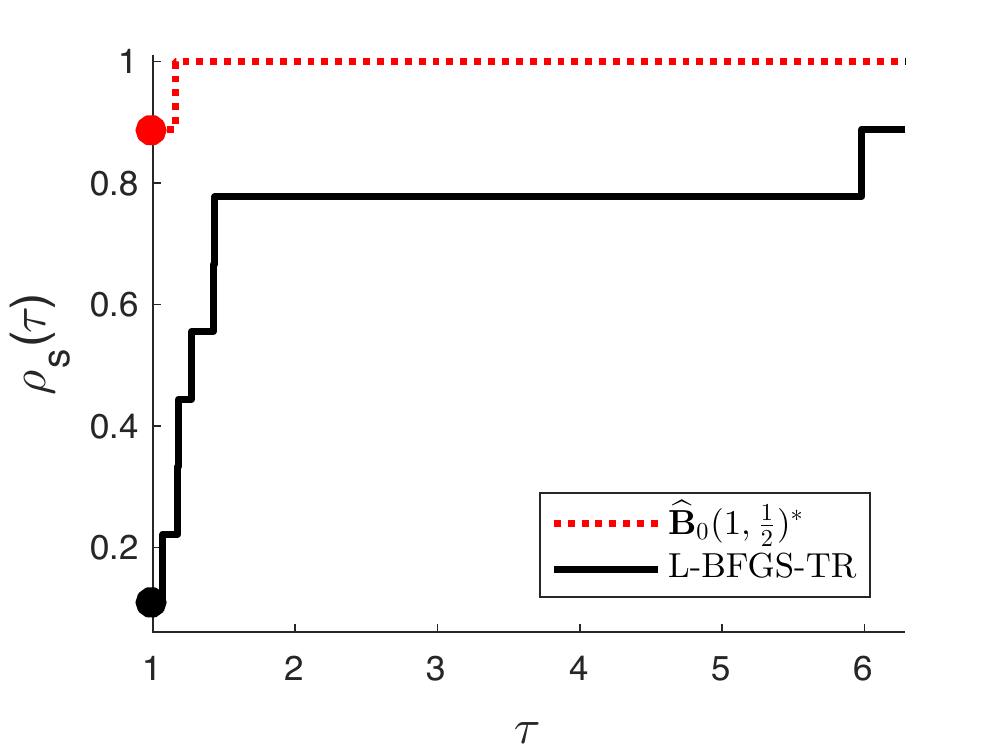}
                                
                        \end{minipage}
                        \hfill
                        \begin{minipage}{0.48\textwidth}
                                                \includegraphics[width=\textwidth]{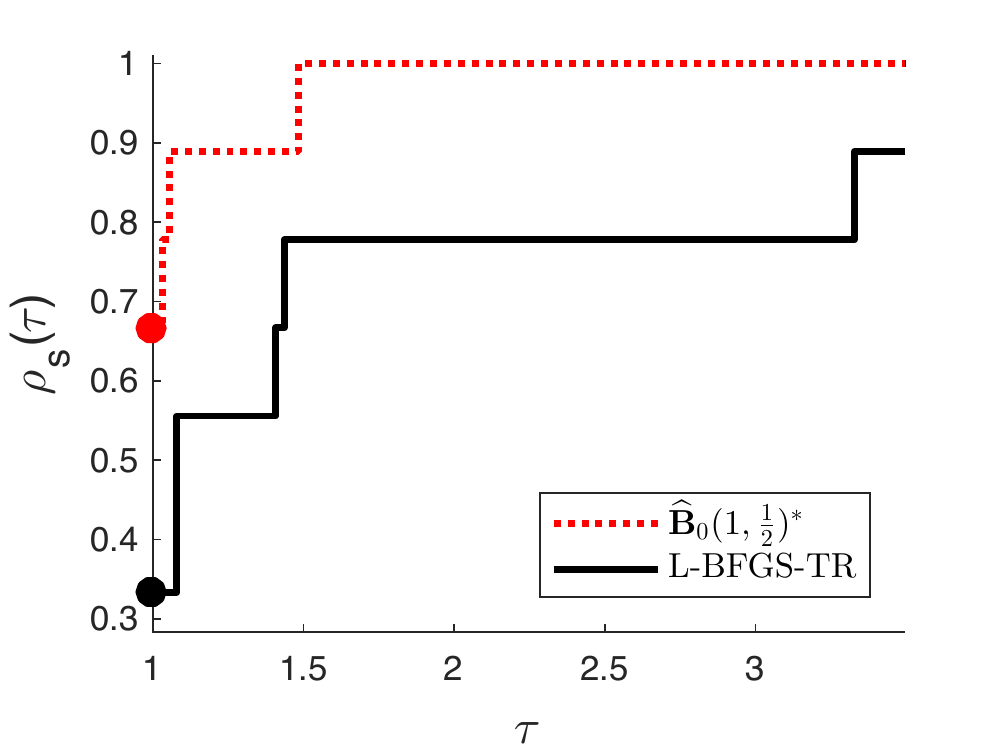}
                                
                        \end{minipage}
                        \caption{Performance profiles of 
                        \texttt{iter} (left) and
\texttt{time} (right) for Experiment 4
                        comparing  LMTR with the dense initialization
with 
$\gamma_k^{\perp}(1,\frac{1}{2})$
 to L-BFGS-TR
on the subset of 14 problems for which L-BFGS-TR implements a line search more than $30\% $  of the iterations.}
                \label{fig:exp6-sel}       
                \end{figure*}

\bigskip
\noindent
{\bf Experiment 5.}  In this experiment, we compare the results of {\small LMTR} using the 
dense initialization to that of {\small LMTR} using the conventional diagonal initialization $B_0=\gamma_k I$ where $\gamma_k$ is given by  (\ref{eqn-diagInit}).  The dense initialization selected
was chosen to be the top performer from Experiment 2 (i.e.,
$\gamma_k^{\perp}(1,\frac{1}{2})$).
%and the {\small QR} factorization
%was used to compute products with $P_\parallel$.

  \begin{figure*}[h!]
                        \begin{minipage}{0.48\textwidth}
                                \includegraphics[width=\textwidth]{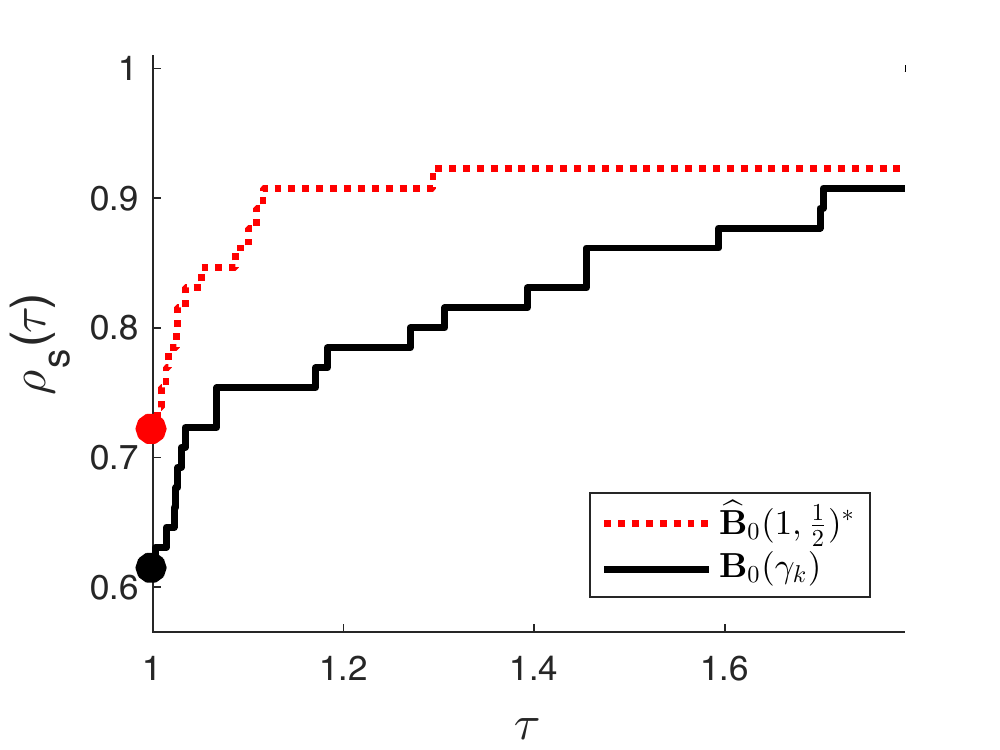}
                                
                        \end{minipage}
                        \hfill
                        \begin{minipage}{0.48\textwidth}
                                                \includegraphics[width=\textwidth]{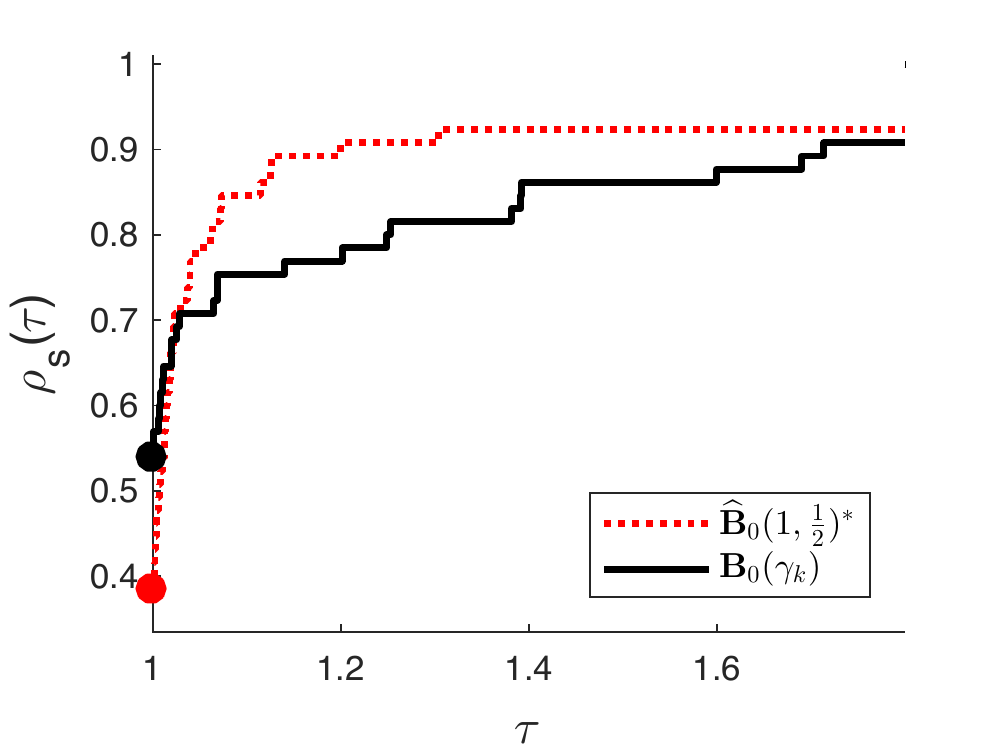}
                                                
                        \end{minipage}
                        \caption{Performance profiles of 
                        \texttt{iter} (left) and
\texttt{time} (right) for Experiment 5 comparing LMTR with the dense initialization
with
$\gamma_k^{\perp}(1,\frac{1}{2})$
 to LMTR with the conventional initialization.}
                \label{fig:compinitial}      
\end{figure*}

  \begin{figure*}[h!]
                        \begin{minipage}{0.48\textwidth}
                                \includegraphics[width=\textwidth]{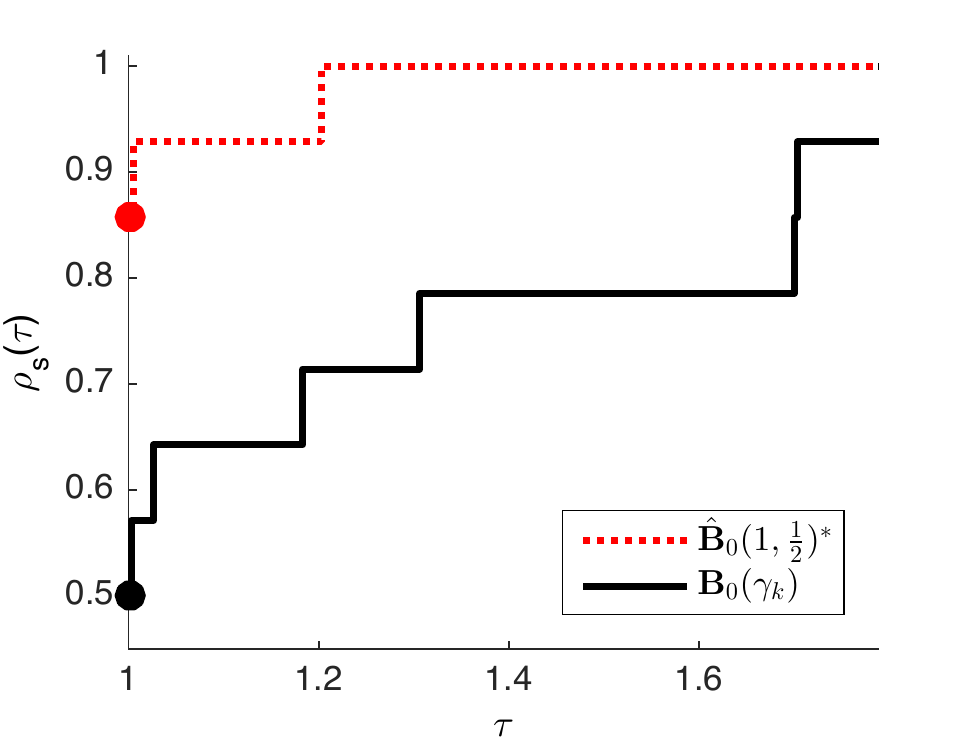}

                        \end{minipage}
                        \hfill
                        \begin{minipage}{0.48\textwidth}
                                                \includegraphics[width=\textwidth]{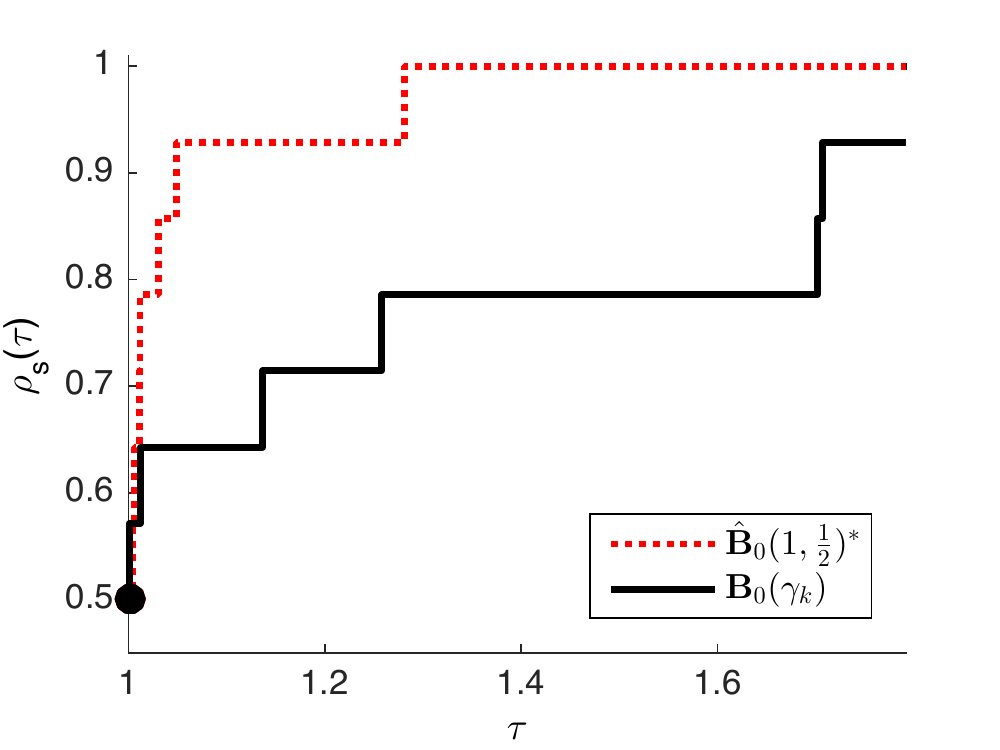}
                        \end{minipage}
                        \caption{Performance profiles of 
                        \texttt{iter} (left) and
\texttt{time} (right) for Experiment 5 comparing LMTR with the dense initialization
with
$\gamma_k^{\perp}(1,\frac{1}{2})$
 to LMTR with the conventional initialization
on the subset of 14 problems in which the unconstrained minimizer is rejected at  $30\% $  of the iterations.}
                \label{fig:compinitial_sel}       
\end{figure*}

From Figure~\ref{fig:compinitial}, the dense initialization with
$\gamma_k^{\perp}(1,\frac{1}{2})$ outperforms the conventional
initialization for
{\small LMTR} in terms of iteration count; however, it is unclear
  whether the algorithm benefits from the dense initialization in terms of
  computational time.  The reason for this is that the dense initialization
  is being used for all aspects of the {\small LMTR} algorithm; in
  particular, it is being used to compute the full quasi-Newton step
  $p_u^*$ (see the discussion in Experiment 1), which is typically accepted
  most iterations on the {\small CUTE}st test set.  Therefore, as in Experiment 5, we
  compared {\small LMTR} with the dense initialization and the conventional
  initialization on the subset of 14 problems in which the unconstrained
  minimizer is rejected at least 30\% of the iterations.  The performance
  profiles associated with this reduced set of problems are found in
  Figure~\ref{fig:compinitial_sel}.  The results from this experiment
  clearly indicate that on these more difficult problems the dense
  initialization outperforms the conventional initialization in both
  iteration count and computational time.

%% file: 5-Conclusion.tex
In this paper, we presented a dense initialization for quasi-Newton methods
to solve unconstrained optimization problems.  This initialization makes
use of two curvature estimates for the underlying function in two
complementary subspaces.  
Importantly, this initialization neither introduces additional
computational cost nor increases storage requirements; moreover,
it maintains theoretical convergence properties of quasi-Newton methods.
It should also be noted that this initialization still
makes it possible to efficiently compute products and perform solves with
the sequence of quasi-Newton matrices.  

The dense initialization is especially well-suited for use in the
shape-changing infinity-norm \LBFGS{} trust-region method. Numerical
results on the 
outperforms both the standard \LBFGS{} line search method as well as the
same shape-changing trust-region method with the conventional
initialization.  Use of this initialization is possible with any
quasi-Newton method for which the update has a compact representation.
While this initialization has broad applications for quasi-Newton line
search and trust-region methods, its use makes most sense from a
computational point of view when the quasi-Newton method already computes
the compact formulation and partial eigendecomposition; if this is not the
case, using the dense initialization will result in additional
computational expense that must be weighed against its benefits.